\documentclass[lettersize,journal]{IEEEtran}
\usepackage{epsfig}
\usepackage{epstopdf}
\usepackage{graphicx,latexsym,amsmath,amssymb,amsthm,enumerate}
\usepackage{amscd}
\usepackage{array}
\usepackage{textcomp}
\usepackage{float}
\usepackage{amsmath}
\usepackage{bbding}
\usepackage{amsfonts}
\usepackage{makecell}
\usepackage{mathrsfs}
\usepackage{tikz}
\usepackage{amssymb}
\usepackage{booktabs}
\usepackage{bm}
\usepackage{caption}
\usepackage{diagbox} 
\usepackage[caption=false,font=normalsize,labelfont=sf,textfont=sf]{subfig}
\usepackage{stfloats}
\usepackage{url}
\usepackage{verbatim}
\usepackage{romannum}
\usepackage[driverfallback=dvipdfm,
           pdfstartview=FitH,
           colorlinks=true,
             pdfborder=001,
           citecolor=blue]{hyperref}
\allowdisplaybreaks[4]

\DeclareTextSymbol{\cyrsftsn}{OT2}{126}
\DeclareTextSymbol{\textnumero}{OT2}{125}

\theoremstyle{definition}
\newtheorem{definition}{Definition}[section]
\newtheorem{theorem}{Theorem}[section]
\newtheorem{lemma}{Lemma}[section]
\newtheorem{corollary}{Corollary}[section]

\newtheorem{remark}{Remark}[section]
\newtheorem{example}{Example}[section]

\newtheorem{assumption}{Assumption}[section]
\newtheorem{notation}{Notation}[section]
\newtheorem{problem}{Problem}
\def\BibTeX{{\rm B\kern-.05em{\sc i\kern-.025em b}\kern-.08em
    T\kern-.1667em\lower.7ex\hbox{E}\kern-.125emX}}
\usepackage{balance}

\begin{document}
\pagenumbering{arabic}

\title{Time-Varying Distributed Optimization for A Class of Stochastic Multi-Agent Systems}
\author{Wan-ying Li, Nan-jing Huang
\thanks{Manuscript created January, 2025; This work was developed by the National Natural Science Foundation of China (12171339, 12471296). (Corresponding author: Nan-jing Huang)}
\thanks{The authors are with the Department of Mathematics, Sichuan University, Chengdu 610064, P.R. China (e-mails: liwanying\_scu@163.com (Wan-ying Li); njhuang@scu.edu.cn and nanjinghuang@hotmail.com (Nan-jing Huang)}}


\maketitle

\begin{abstract}
Distributed optimization problems have received much attention due to their privacy preservation, parallel computation, less communication, and strong robustness. This paper presents and studies the time-varying distributed optimization problem for a class of stochastic multi-agent systems for the first time. For this, we initially propose a protocol in the centralized case that allows the tracking error of the agent with respect to the optimal trajectory to be exponentially ultimately bounded in a mean-square sense by stochastic Lyapunov theory. We then generalize this to the distributed case. Therein, the global variable can be accurately estimated in a fixed-time by our proposed estimator. Based on this estimator, we design a new distributed protocol, and the results demonstrate that the tracking error of all agents with respect to the optimal trajectory is exponentially ultimately bound in a mean-square sense by stochastic Lyapunov theory. Finally, simulation experiments are conducted to validate the findings.
\end{abstract}


\section{Introduction}
\IEEEPARstart{R}{ecently}, multi-agent systems (MASs) have developed rapidly, in which the distributed control problem has received much attention, with the advantages of privacy preservation, parallel computation, less communication, and robustness. The difficulty in solving these problems is that only local information and limited information from neighboring agents can be used to reach a decision. Among them, distributed optimization problems are an important class of problems with widespread applications, such as smart grids \cite{2019-mao}, sensor networks \cite{2004-rabbat}, finance \cite{2021-bae}, etc.

Scholars have studied distributed optimization problems in depth from several aspects. A distributed optimization algorithm for minimizing a common cost function was first proposed by Tsitsiklis et al.\cite{1986-Tsi} Subsequently, scholars have provided frameworks for solving distributed optimization problems using different methods, such as subgradient methods \cite{2009-nedic}, zero-gradient-sum methods \cite{2012-lu}, etc.
Later, the works \cite{2016-lin,2016-yang,2018-wang} considered a class of distributed optimization problems with communication delays;
\cite{2015-wang,2022-liu,2023-yu} considered external disturbances and proposed corresponding protocols to address such distributed optimization problems.
Further, an event-triggered mechanism is proposed in \cite{2021-li,2021-song} to reduce the communication requirements from the perspective of reducing the communication cost and solve the distributed optimization problem.
Recently, \cite{2018-chen,2020-wang,2021-yu,2023-garg,2021-song} proposed protocols to improve the convergence speed, ensuring convergence in a finite/fixed time.
More results can be traced back to \cite{2010-wang,2018-nedic,2019-yang} and the references therein.

It is notable that the protocols referenced in the above works are all for solving distributed time-invariant optimization problems (TIV-OPs), where the objective functions does not depend on time.
Indeed, it is both natural and essential to consider distributed time-varying optimization problems (TV-OPs) because their objective functions adjust over time, which is more realistic.
So far, several studies considering the time-varying objective functions have been conducted based on different methods, such as the method based on sign function \cite{2016-rahili}, sliding mode control method \cite{2017-sun}, output regulation method \cite{2022-ding}, prediction-correction method \cite{2016-simon}, etc.
Subsequently, scholars have extended the distributed TV-OPs with inequality constraints \cite{2020-sun} and coupled equality constraint \cite{2024-yue}.
The rest of the research results in this area can be found in \cite{2017-ning,2024-zhu,2020-li,2024-li} and their references.

In the aforementioned studies, all MASs involved are described by ordinary differential equations (ODEs). However, in practical applications, the evolution of systems is inevitably subject to various stochastic disturbances and uncertainties, which significantly impact the system's stability and cannot be modeled by ODEs. Therefore, it is crucial and meaningful to incorporate stochastic factors into the study of distributed TV-OPs. In related researches, stochastic differential equations (SDEs) are commonly used to describe these stochastic disturbances and uncertainties, and such MASs are referred to as stochastic MASs (SMASs). As far as the author is aware, current researches on the SMASs are primarily focused on consensus \cite{2019-you,2024-tang}, with no relevant studies on distributed TV-OPs of the SMASs to date. Thus, it would be important and interesting to study the distributed TV-OPs based on the SMASs under some suitable conditions.

The present work is an attempt in this new direction. Compared to the MASs described by ODEs, the difficulty is the treatment of stochasticity, which leads to the failure of various existing time-varying optimization protocols. Instead, one needs to innovatively find some new protocols for ensuring the tracking error of the agents with respect to the optimal trajectory to be globally exponentially ultimately bound in a mean-square sense (ME-GEUB). The main contributions of this paper are four-fold. First, the centralized and distributed TV-OPs for the SMASs are proposed for the first time, which provide a broader practical context than the distributed TIV-OPs. Second, a proper protocol is proposed to solve the centralized TV-OPs of the SMASs by stochastic Lyapunov theory. Third, a global estimator is designed to deal with the treatment of stochasticity and we propose a new protocol to address the distributed TV-OPs of the SMASs by a new Lyapunov criteria. Last but not least, some numerical simulations are given to validate the findings.

The remainder that follows in this paper is structured below. The next section reviews some basic preliminaries. Subsequently, in Section \Romannum{3}, a centralized protocol is designed to deal with the centralized TV-OPs. The centralized TV-OPs are then extended to the distributed one and the corresponding distributed protocol is presented in Section \Romannum{4}. Some simulation experiments are provided in Section \Romannum{5} and then the conclusions of the paper are summarized in Section \Romannum{6}.

\section{Preliminaries}
\subsection{Graph theory and Notations}
\noindent Consider an interaction graph $\mathcal{G}$ with $N$ nodes with index set $\mathfrak{N}=\{1,2,\cdots,N\}$, where each node represents an agent. In the subsequent discussion, we will frequently use the two matrices of the interaction graph $\mathcal{G}$ for our analysis. One of them is the adjacency matrix $\mathcal{A} = [a_{ij}]_{n\times n}$ with $a_{ij} > 0$ if agent $i$ can receive the information from agent $j$; if not, $a_{ij} = 0$. Generally, we always set $a_{ii} = 0$. In addition, $\mathcal{G}$ is called an undirected graph if $a_{ij} = a_{ji}$, otherwise it is a directed graph. The other matrix we need to use is the Laplacian matrix ($\mathcal{L}$-matrix) of $\mathcal{G}$ denoted as $\mathcal{L} = [l_{ij}]_{n\times n}$ ($l_{ij}=-a_{ij}$ for $i \neq j$, $l_{ii} = \sum_{j=1}^{n}a_{ij}$). The main notations used subsequently are given by Table \ref{Notations}.
\begin{table}[h]
  \caption{Notations} \label{Notations}
  \begin{center}
  \begin{tabular}{|p{1.1cm}|p{6.9cm}|}
  \hline
  Notation &Meaning  \\
  \hline
  $\mathbb{R}$ / $\mathbb{R}_+$/$\mathbb{R}_{++}$   &   $(-\infty,\infty)$ / $[0,\infty)$ / $(0,\infty)$  \\
  \hline
  $\mathbb{R}^n$   & $n$-dimensional Euclidean space  \\
  \hline
  $\nabla f(t,x)$   &  the gradient of $f(t,x)$ at $x$ \\
  \hline
  $H_f(t,x)$   & the Hessian matrix of $f(t,x)$ at $x$ \\
  \hline
  $\nabla_{xt} f(t,x)$   & the partial derivative of $\nabla f(t,x)$ at $t$ \\
  \hline
  $\otimes$  & Kronecker product \\
  \hline
  $\lambda_{i}[A]$  & The $i$-th eigenvalue after ranking the eigenvalues of matrix $A$ from smallest to largest 
  \\
  \hline
  $\|x\|_p$  & the $p$-norm of vector $x\in \mathbb{R}^n$, $\|x\|_{p}=(\sum_{i=1}^{m}|x_{i}|^{p})^{\frac{1}{p}}$ \\
  \hline
  $\|A\|$  & the 2-norm of matrix $A\in \mathbb{R}^{n\times n}$, $\|A\|=\lambda_{\max}[A]$ \\
  \hline
  $\|A\|_F$  & the Frobenius-norm of $A\in \mathbb{R}^{n\times n}$, $\|A\|_F=\sqrt{\sum_{i,j=1}^{n}a_{ij}^2}=\sqrt{trace[A^TA]}$\\
  \hline
  $\mathcal{C}^2$& the class of functions $V(t,x)$ twice continuously differentiable with respect to $x$ and once continuously differentiable with respect to $t$ \\
  \hline
  $(\Omega, \mathcal{F}, P)$ & a complete probability space, where $\Omega$ is a sample space, $\mathcal{F}$ is a $\sigma$-field and $P$ is a probability measure.\\
  \hline
  \end{tabular}
  \end{center}
\end{table}
For $x=[x_1,x_2,\cdots,x_n]^T\in \mathbb{R}^n$ and $m>0$, $sig^m(x)=[sig^m(x_{1}), sig^m(x_{2}), \cdots , sig^m(x_{n})]^{T}$, where $sig^m(x_{i})=sign(x_{i})|x_{i}|^m$ and
$sign(x_{i}) =
\begin{cases}
  1, & \mbox{if $x_{i}>0$}; \\
  -1, & \mbox{if $x_{i}<0$}; \\
  0, & \mbox{if $x_{i}=0$}.
\end{cases}
$

\subsection{Basic knowledge}\noindent

\noindent Below we present several definitions and lemmas as preliminaries for the ensuing discussion.
\begin{definition}\label{def-graph}
For a directed graph $\mathcal{G}$ there are the following definitions:
\begin{enumerate}
\item{If there exists $\xi_i, \xi_j\in \mathbb{R}_{++}$ such that $\xi_ia_{ij}=\xi_ja_{ji}$, $\forall i, j \in \Gamma$, then $\mathcal{G}$ is said to be detail balanced in weights.}
\item{If there exists a path between any pair of nodes in $\mathcal{G}$, then $\mathcal{G}$ is said strongly connected.}
\end{enumerate}
\end{definition}
\begin{lemma}\cite{D2.2,L2.3}\label{lemma-eig}
For an undirected connected graph $\mathcal{G}$, its Laplacian matrix $\mathcal{L}$ is semipositive definite and has n non-negative eigevalues $0=\lambda_1\leq\lambda_2\leq\dots\leq\lambda_n$. Moreover, $\lambda_2$ and $\lambda_n$ satisfy
  \begin{equation*}
    \lambda_{2}[\mathcal{L}]=\min_{v\neq0, \textbf{1}^Tv=0}\frac{v^T\mathcal{L}v}{v^Tv}, \quad \lambda_{n}[\mathcal{L}]=\max_{v\neq0, \textbf{1}^Tv=0}\frac{v^T\mathcal{L}v}{v^Tv}.
  \end{equation*}
\end{lemma}
\begin{definition}\cite{convex}\label{def-Sconvex}
  A $\mathcal{C}^2$-function $f(t,x)$ is called strongly convex of $x$, if there exists a constant $L>0$ such that 
  for any $x,y\in \mathbb{R}^n$, we have
  \[ f(t,y) - f(t,x)-\langle \nabla f(t,x), y-x\rangle \geq \frac{L}{2}\|y-x\|^2.\]
\end{definition}
\begin{lemma}\cite{2018-2}\label{lemma-min}
Suppose $f(t,x)$ is a $\mathcal{C}^2$-function satisfying $\nabla f(\bar{t},\bar{x})=0$ and $\nabla^2 f(\bar{t},\bar{x})\succ 0$. If $\bar{x}$ is an isolated strict local minimizer of $f(t,x)$ at time $\bar{t} \in \mathbb{R}_+$, then there exists a $C^1$-function $x:B_{\varepsilon}(\bar{t})\rightarrow \mathbb{R}^n$ where $B_{\varepsilon}(\bar{t})$ is a neighborhood of $\bar{t}$ $(\varepsilon>0)$, such that $x(\bar{t})=\bar{x}$ and $x(t)$ is an isolated strict local minimizer of $f(t,x)$ for any $t\in B_{\varepsilon}(\bar{t})$.
\end{lemma}

\begin{lemma}\cite{L2.2,L2.2+1,fix4}\label{lemma-inequality}
The following inequalities hold:
\begin{enumerate}
\item{for $x \in \mathbb{R}^n$, $\|x\|_q \leq \|x\|_p \leq n^{{\frac{1}{p}}-\frac{1}{q}}\|x\|_q$, where $q>p>0$;}
\item{for $A, B\in \mathbb{R}^{n\times n}$, $\frac{1}{\sqrt{n}}\|A\|_1\leq \|A\| \leq \sqrt{n}\|A\|_1$, $\|A\|\leq \|A\|_F \leq \sqrt{n}\|A\|$, $\|AB\|\leq\|A\|\cdot\|B\|$;}
\item{for $x_i\in \mathbb{R}_+$, one has
\begin{align*}
&\left(\sum_{i\in \mathfrak{N}}x_{i}\right)^{k} \leq \sum_{i\in \mathfrak{N}}x_{i}^{k}\leq N^{1-k}\left(\sum_{i\in \mathfrak{N}}x_{i}\right)^{k}\;(0<k<1),\\
&\sum_{i\in \mathfrak{N}}x_{i}^{k}\leq \left(\sum_{i\in \mathfrak{N}}x_{i}\right)^{k}\leq N^{k-1}\sum_{i\in \mathfrak{N}}x_{i}^{k}\;(k\geq1).
\end{align*}}
\end{enumerate}
\end{lemma}

\begin{lemma}\cite{Matrix-Analysis}\label{weyl}(Weyl)
  Let $A,B\in\mathbb{R}^{n\times n}$ be positive definite symmetric and let the respective eigenvalues of $A$, $B$ and $A+B$ be $\{\lambda_i[A]\}_{i=1}^{n}$, $\{\lambda_i[B]\}_{i=1}^{n}$ and $\{\lambda_i[A+B]\}_{i=1}^{n}$ (the sequence has been ordered from the smallest to the largest).
  Then
\[\lambda_i[A+B]\leq\lambda_{i+j}[A]+\lambda_{n-j}[B], \, j=0,1,\ldots,n-i,\]
for each $i=0,1,\ldots,n$. For some pair $i,j$, the condition for the equality to hold is that there is a nonzero vector $x$ such that $Ax=\lambda_{i-j+1}[A]x$, $Bx=\lambda_{j}[B]x$ and $(A+B)x=\lambda_{i}[A+B]x$. Otherwise, the above inequality is a strict inequality. Also
\[\lambda_{i-j+1}[A]+\lambda_{j}[B]\leq\lambda_{i}[A+B], \, j=1,\ldots,i,\]
for each $i=0,1,\ldots,n$. The conditions under which the equality holds or not are the same as above. 
\end{lemma}
\begin{corollary}
  Let $A,B\in\mathbb{R}^{n\times n}$ be positive definite symmetric. Then
  \[\lambda_i[A]+\lambda_1[B]\leq\lambda_i[A+B]\leq\lambda_i[A]+\lambda_n[B], \, i=0,1,\ldots,n.\]
  It can be further obtained that
  \[\lambda_1[A+B]\geq\lambda_1[A]+\lambda_1[B], \, \lambda_n[A+B]\leq\lambda_n[A]+\lambda_n[B].\]
\end{corollary}

\subsection{Stability}
\noindent This subsection reviews some definitions and theorems of stability.
We first consider the following MASs
\begin{equation}\label{MAS}
\dot{x}_i(t)=\phi(t,x_i(t)), \quad x_i(0)=x_{i0}.
\end{equation}
\begin{definition}\cite{2017-ning}\label{def-FXC}
The fixed-time consensus is said to be achieved, if there exists $T\in \mathbb{R}_{++}$ such that, for all $i,j\in \mathfrak{N}$ $\lim_{t\rightarrow T}\|x_{i}(t)-x_{j}(t)\|=0$ and $\|x_{i}(t)-x_{j}(t)\|=0$ hold for all $t\geq T$.
\end{definition}
\begin{lemma}\cite{D2.4,fix8}\label{lemma-FXC}
For a system \eqref{MAS}, if a Lyapunov function $V\in \mathcal{C}^2$ is continuous radially unbounded and any solution $x(t)$ satisfied $\dot{V}(t,x)\leq -k_1V(t,x)^{p}-k_2V(t,x)^{q}$, where $k_1, k_2\in\mathbb{R}_
  {++}$, $p\in (0,1)$ and $q\in (1, +\infty)$. Then fixed-time stability is said to be reached with the settling time $T\leq T_{\max}=\frac{1}{k_1(1-p)}+\frac{1}{k_2(q-1)}$.
\end{lemma}

We also consider the following stochastic differential equation
\begin{equation}
d x_t = b(t,x_t) dt + \sigma(t,x_t)dB(t),\quad
x_0 = s_0,
\label{SDE}
\end{equation}
where $x: \mathbb{R}_+\times\Omega\rightarrow \mathbb{R}^n$ is the state denoted as $x_t$, $s_0: \Omega\rightarrow \mathbb{R}^n$ is the initial state with $\mathbb{E}(\|s_0\|^2)<\infty$, $b:\mathbb{R}_+\times \mathbb{R}^n \rightarrow \mathbb{R}^n$, $\sigma:\mathbb{R}_+ \times \mathbb{R}^n \rightarrow \mathbb{R}^{n\times n}$, and $B(t)$ is an $n$-dimensional Brownian motion defined on $(\Omega, \mathcal{F}, P)$.

\begin{definition}\cite{SSDE}\label{def-LV}
  For any $V\in \mathcal{C}^2$ and the system \eqref{SDE}, the differential operator $\mathcal{L}$ acting on $V$ is defined as
  \begin{align*}
  \mathcal{L}V(t,x) = &\frac{\partial V(t,x)}{\partial t}+ \frac{\partial V(t,x)}{\partial x}b(t,x)\nonumber\\
  &+\frac{1}{2}trace\left[\sigma^T(t,x)\frac{\partial^2 V(t,x)}{\partial x^2}\sigma(t,x)\right].\nonumber
  \end{align*}
\end{definition}
\begin{lemma}\cite{2001-deng}\label{lemma-bound}
  There exists a positive definite function $V\in \mathcal{C}^2$, $k_1,k_2\in\mathbb{R}_
  {++}$, and class $\mathcal{K}_{\infty}$ function $\bar{\alpha}_1,\,\bar{\alpha}_2$ such that $\bar{\alpha}_1(|x|)\leq V(x)\leq \bar{\alpha}_2(|x|)$ and $\mathcal{L}V(x)\leq -k_1 V(x)+k_2$ for all $x\in\mathbb{R}^n$, $t>t_0$. Then, for each $x_0\in\mathbb{R}^n$, one gets
   \begin{enumerate}
\item{the system \eqref{SDE} has a unique strong solution;}
\item{$\mathbb{E}[V(x)]\leq V(x_0)\exp\{-k_1 t\}+k_2/k_1, \, \forall t>0$.}
\end{enumerate}

\end{lemma}
\begin{definition}\cite{2023-min}\label{def-Sfx}
  The solution $x(t;x_0)$ of the system \eqref{SDE} is said to be practically fixed-time stability (Pfxs) in probability, if
\begin{enumerate}
\item{the system \eqref{SDE} has a unique solution for all initial state $ x_0\in \mathbb{R}^n$;}
\item{for every initial state $x_0\in \mathbb{R}^n\setminus \{0\}$, the settling time $\tau=\inf\{t: x(t;x_0)\in \{0\}\}$ is the first time $\{0\}$ is reached and it satisfies $\mathbb{E}(\tau)\leq T$. In particular, $T$ does not depend on the initial value;}
\item{for a constant $\epsilon_0>0$, $\mathbb{E}(\|x(t;x_0)\|)\leq\epsilon_0$ holds for all $t\geq\tau$.}
\end{enumerate}
\end{definition}

\begin{lemma}\cite{2023-min}\label{lemma-Sfxs}
  For the system \eqref{SDE}, suppose $b(t,x), \, \sigma(t,x)$ satisfy the monotone consition and are locally Lipschitz continuous in $x$. If there exists a continuous radially unbounded $V\in\mathcal{C}^2$ and constants $k_1,k_2\in\mathbb{R}_
  {++},\,p\in(0,1),\,q\in(1,+\infty), \,\triangle>0$ such that
  \[\mathcal{L}V(x)\leq -k_1V(x)^{p}-k_2V(x)^{q}+\triangle, \, \forall x\in\mathbb{R}^n,\]
  then the solution of system \eqref{SDE} is Pfxs in probability and the stochastic settling time $\tau = \inf \{t:\mathbb{E}(V(x))\leq \delta\}$ satisfies
  \[\mathbb{E}(\tau)\leq \frac{(\bar{k}_1/\bar{k}_2)^{\frac{1-p}{q-p}}}{\bar{k}_1(1-p)}+ \frac{(\bar{k}_1/\bar{k}_2)^{\frac{1-q}{q-p}}}{\bar{k}_2(q-1)}=T_{\max},\, \forall x_0\in \mathbb{R}^n\backslash\{0\},\]
  where $\bar{k}_1=k_1-\frac{\triangle}{\delta^p}>0$, $\bar{k}_2=k_2-\frac{\triangle}{\delta^q}>0$ and $\delta=\min\left\{\left(\frac{\triangle}{(1-\kappa)k_1}\right)^{\frac{1}{p}},\left(\frac{\triangle}{(1-\kappa)k_2}\right)^{\frac{1}{q}}\right\}$, and $0<\kappa<1$ is a design constant.
\end{lemma}

\begin{definition}\cite{2020-he}\label{def-MS-GEUB}
  Let $\gamma,\,\lambda,\,\rho$ be given positive constants. Systems \eqref{SDE} is said to be MS-GEUB if for any initial value $x_0$, the solution $x(t;x_0)$ of \eqref{SDE} satisfies $\mathbb{E}(\|x(t;x_0)\|^2)\leq \gamma \exp\{-\lambda t\}\mathbb{E}(\|x_0\|^2)+\rho$ for all $t>0$.
\end{definition}

\section{Centralized optimization problem}
\noindent We first address the following optimization problem in a centralized framework.
\begin{problem}(Centralized TV-OPs)\label{P1}
\begin{equation}
\left\{
\begin{array}{l}
\min \mathbb{E}[F(t,x_t)]\\
 d x_t = u^c dt + \sigma(t,x_t)dB(t),\,
x_0 = s_0,
\end{array}
\right.
\label{cx}
\end{equation}
  where $F:\mathbb{R}_+ \times \mathbb{R}^n \rightarrow \mathbb{R}$ and the relevant assumptions for \eqref{cx} are the same as for \eqref{SDE}, $u^c$ is the protocol which will be given below.
\end{problem}
\begin{remark}
  In Problem \ref{P1}, we present a stochastic version of the centralized optimization problem. Non-stochastic version has been proposed and studied in \cite{2016-rahili,fix4}. In comparison, we model our problem using the SDE, which takes the stochastic factor into account and has a much wider application. Below we design a suitable protocol $u^c$, under which there exists a sufficiently large $T$ such that for any $t\geq T$, the agent states $x_t$ satisfies that $\mathbb{E}\|x_t-x^*_t\|^2$ is bounded, where $x^*_t$ is the optimal trajectory.
\end{remark}
We now introduce a few assumptions that need to be used in this section:
\begin{assumption}\label{C-A1}
 The diffusion term in \eqref{cx} satisfies $\|\sigma(t,x)\|_F\leq \bar{\sigma}$ for all $x\in\mathbb{R}^n$.
\end{assumption}

\begin{assumption}\label{C-A2}
The objective function $F(t,x)$ satisfies
\begin{enumerate}
\item{$l_1$-strongly convex on $x$;}
\item{$\left\|\nabla_{xt} F(t,x) - \nabla_{xt} F(t,y)\right\|\leq l_2\|x-y\|$;}
\item{$H_F(t,x)$ is a positive definite invertible matrix and $\lambda_{\min}(H_F(t,x))\geq h>0$. In addition, $
H_F(t,x) = H_F(t,y)$ for all $x,y\in \mathbb{R}^n$.}
\end{enumerate}
\end{assumption}
\begin{remark}
  Assumption \ref{C-A1} presents the boundedness for the diffusion term and has been used in the \cite{2019-xu}. A bounded diffusion term ensures that the energy of the system or the deviation of the solution is within a certain range, thus helping to analyze the long-term behavior and stability of the system. Assumption \ref{C-A2} (i) of strong convexity ensures the existence of the optimal trajectory of Problem \ref{P1}, (ii) and (iii) are common assumptions in TV-OPs \cite{2016-rahili,2016-simon,2020-sun,2017-ning} where the assumption on $ \nabla_{xt} F(t,x)$ is bounded, here we only need lipschitz continuity.
\end{remark}
From the Lemmas \ref{lemma-min} and the Assumption \ref{C-A2}, clearly, there exists a unique optimal trajectory of Problem \ref{P1}. Then we construct the following centralized protocol:
\begin{equation}\label{u_c}
u^c = -\gamma_1 \nabla F(t,x_t)- H_F(t,x_t)^{-1}\nabla_{xt} F(t,x_t),
\end{equation}
where $\gamma_1>0.$ Using the above protocol, we can obtain the following theorem.
\begin{theorem}\label{the-cen}
With the Assumptions \ref{C-A1}-\ref{C-A2} and $\gamma_1>\frac{l_2}{hl_1}$, under protocol \eqref{u_c}, the tracking error of the agent with respect to the optimal trajectory $x^*_t$ is MS-GEUB, i.e.
\begin{align*}
\mathbb{E}\|x_t-x^*_t\|^2\leq &\mathbb{E}\|x_0-x^*_0\|^2\exp\{-(2\gamma_1 l_1-\frac{2l_2}{h})t\}\\
&+\frac{\bar{\sigma}^2h}{2\gamma_1 l_1h-2l_2}
\end{align*}
for all $t>0$.
\end{theorem}
\begin{proof}
From Lemma \ref{lemma-min} and Assumption \ref{C-A2}, there exists
\[x^*_t=\arg\min F(t,x)
   \Leftrightarrow\nabla F(t,x^*_t)=0.\]
   Differentiating both sides of the equation on $t$ yields
  \begin{equation}\label{y*}
    dx^*_t =  -H_F(t,x^*_t)^{-1}\nabla_{xt}F(t,x^*_t)dt.
  \end{equation}
Based on the system \eqref{y*}, we establish the following error system where $\hat{x}_t = x_t-x^*_t$,
\[
d\hat{x}_t
= \left(u^c + H_F(t,x^*_t)^{-1}\nabla_{xt} F(t,x^*_t)\right)dt+\sigma(t,x_t)dB(t).
\]
Consider a positive definite function $V_1 = \frac{1}{2}\hat{x}_t^T\hat{x}_t$, based on Definition \ref{def-LV} we have
\begin{align*}\label{LV1}\small
&\mathcal{L}V_1\\
=& \hat{x}_t^T\left(-\gamma_1 \nabla F(t,x_t)-H_F(t,x_t)^{-1}\nabla_{xt} F(t,x_t)\right.\\
&\left.+ H_F(t,x_t^*)^{-1}\nabla_{xt}F(t,x_t^*)\right)
+ \frac{1}{2}trace[\sigma(t,x_t)^T\sigma(t,x_t)]\\
\leq& \frac{\bar{\sigma}^2}{2} \underbrace{-\gamma_1 \hat{x}_t^T \nabla F(t,x_t)}_{W_1}\\
&\underbrace{-\hat{x}_t^T\left(H_F(t,x_t)^{-1}\nabla_{xt} F(t,x_t) - H_F(t,x_t^*)^{-1}\nabla_{xt} F(t,x_t^*)\right)}_{W_2}.
\end{align*}
By the properties of strongly convex functions in Definition \ref{def-Sconvex}, we have
\begin{align*}
W_1=&-\gamma_1 \hat{x}_t^T \nabla F(t,x_t)
=-\gamma_1 \hat{x}_t^T (\nabla F(t,x_t)-\nabla F(t,x_t^*))\\
\leq& -\gamma_1 l_1 \|\hat{x}_t\|^2 = -2\gamma_1 l_1 V_1.
\end{align*}
Applying Lemma \ref{lemma-inequality} and Assumption \ref{C-A2}, one obtains
\begin{align*}
W_2
=&-\hat{x}_t^T H_F(t,x_t)^{-1}\left(\nabla_{xt}F(t,x_t) -\nabla_{xt}F(t,x_t^*)\right)\\
\leq& \|\hat{x}_t\|\cdot\left\|H_F(t,x_t)^{-1}\right\|\cdot\left\|\nabla_{xt}F(t,x_t) -\nabla_{xt}F(t,x_t^*)\right\|\\
\leq&\frac{l_2}{h}\|\hat{x}_t\|^2=\frac{2l_2}{h}V_1.
\end{align*}
With the above inequalities, we have $\mathcal{L}V_1\leq -(2\gamma_1 l_1-\frac{2l_2}{h})V_1+ \frac{1}{2}\bar{\sigma}^2$.   Then from Lemma \ref{lemma-bound} we can conclude that the agent states will converge to the nearby of the optimal trajectory and for all $t>0$
  \[0\leq \mathbb{E}[V_1]\leq \mathbb{E}[V_1(x_0)]\exp\{-(2\gamma_1 l_1-\frac{2l_2}{h})t\}+\frac{\bar{\sigma}^2h}{4\gamma_1 l_1h-4l_2}.\]
  Then from the above inequality we get that $E[V_1]$ is MS-GEUB, i.e.
 \begin{align*}
\mathbb{E}\|x_t-x^*_t\|^2\leq &\mathbb{E}\|x_0-x^*_0\|^2\exp\{-(2\gamma_1 l_1-\frac{2l_2}{h})t\}\\
&+\frac{\bar{\sigma}^2h}{2\gamma_1 l_1h-2l_2}
\end{align*}
  for all $t>0$.
 \end{proof}
 \begin{remark}
  Theorem \ref{the-cen} is a stochastic version of Theorem 3.4 in \cite{2016-rahili}, in which the centralized TV-OPs on MASs is extended to the one on SMASs.
 \end{remark}

\section{Distributed optimization problem}
\noindent As a matter of fact, the centralized protocols have disadvantages such as high communication costs and poor robustness, however in practical cases, we would prefer to use less expenses for achieving the desired goals with high robustness. We therefore propose the following problem:
\begin{problem}(Distributed TV-OPs with Consensus Constraint)\label{P2}
\begin{equation}
\left\{
\begin{array}{l}
\min \mathbb{E}\left[\sum_{i\in \mathfrak{N}}f_i(t,x_{it})\right],\\
d x_{it} = u_i^d dt + \sigma_i(t,x_{it})dB(t), \, x_{i0} = s_{i0},\\
\mbox{subject to} \, \mathbb{E}\left\|x_{it}-\frac{1}{N}\sum_{j\in \mathfrak{N}}x_{jt}\right\|^2\leq\delta,\,\forall i\in \mathfrak{N},
\end{array}
\right.
\label{x_i}
\end{equation}
  where the assumptions for $x_{it}$, $\sigma_i$ and $s_{i0}$ are the same as \eqref{SDE}, each $f_i:\mathbb{R}_+ \times \mathbb{R}^n \rightarrow \mathbb{R}$, $u_i^d$ is the protocol for each agent i which will be given below, $\delta >0$ is a constant which can be calculated in the following discussion. $\mathbb{E}\|x_{it}-\frac{1}{N}\sum_{j\in \mathfrak{N}}x_{jt}\|^2\leq\delta$ is a weaker version of the consensus constraint, which implies that each agent is not very far from the others. We abbreviate $\sigma_i(t,x_{it})$ as $\sigma_i$ below satisfying Assumption \ref{C-A1}.
\end{problem}
\begin{remark}
The non-stochastic version of the distributed TV-OPs has been mentioned in the literature referred to in Section \Romannum{1}. The problem we proposed is more extensive than that due to the fact that we consider stochastic factors.
Furthermore, the consensus constraint, often written as $x_{it}=x_{jt}, \, \forall i,j\in \mathfrak{N}$ in the non-stochastic version, is hard to reach in the stochastic case. So we approximate the consensus constraint by requiring the mean square deviation of all agent states to reach some small area, which can also be written as $\mathbb{E}\|x_{it}-x_{jt}\|^2\leq\delta$.
In the following we design a suitable protocol $u_i^d$ for each agent i, under which there exists a sufficiently large $T$ such that for any $t\geq T$, the agents' states $x_{it}$ satisfy that $\mathbb{E}\left\|x_{it}-\frac{1}{N}\sum_{j\in \mathfrak{N}}x_{jt}\right\|^2$ and $\mathbb{E}\|x_{it}-x_t^*\|^2$ are all bounded, where $x^*_t$ is the optimal trajectory.
\end{remark}
In the distributed case, each agent only knows the information about itself and its neighbors, but the global information is absent, which leads to difficulties in control. With this in mind, we first propose an estimator for estimating the global information based on the following assumptions.
\begin{assumption}\label{A1}
The graph $\mathcal{G}$ of the SMASs \eqref{x_i} is directed, detail-balanced in weights and strongly connected. The adjacency matrix of $\mathcal{G}$ is denoted as $\mathcal{A} = [a_{ij}]_{n\times n}$ and satisfies Definition \ref{def-graph} with $\tilde{a}_{ij}=\xi_i a_{ij},\, \xi_i\in\mathbb{R}_{++}$.
\end{assumption}
\begin{assumption}\label{A2}
For the Hessian matrix of $f_i(t,x)$ denoted as $H_i(t,x)$, we suppose that the derivative of $H_i(t,x)$ with respect to $t$ satisfies $\sup_{t\in\mathbb{R}_+}\left\|\frac{d}{dt}H_{i}(t,x)-\frac{d}{d t} H_{j}(t,y)\right\|\leq L_H$ for all $i,j\in \mathfrak{N}$, $x,y\in\mathbb{R}^n$,
\end{assumption}
\begin{remark}
Assumption \ref{A1} is a common assumption for a class of directed graphs \cite{2024-li,2023-yu}, of which undirected graphs are a special case. Assumption \ref{A2}  is used to ensure that the estimator estimates accurately in bounded time, which has been used in works such as \cite{2016-rahili,2020-li,fix4,2017-ning}. Whereas they all assume boundedness about the first order partial derivative of  $\nabla f_{i}(t,x)$, $H_{i}(t,x)$ at $t$ and the second order partial derivative of $\nabla f_{i}(t,x)$ at t, we only take assumption on $H_{i}(t,x)$ to be sufficient.
\end{remark}
\begin{lemma}\label{lemma-0}
 Suppose that $g(x)$ is an odd mapping and $a_{ij}=a_{ji}$, we can get
 $\sum_{i,j\in \mathfrak{N}}a_{ij}g(x_i-x_j) = 0.$
\end{lemma}
\begin{proof}
According to the above conditions, we can get the following transformations:
\begin{align*}
  &\sum_{i,j\in \mathfrak{N}}a_{ij}g(x_i-x_j)\\
  =& \frac{1}{2}\sum_{i,j\in \mathfrak{N}}a_{ij}g(x_i-x_j)+\frac{1}{2}\sum_{i,j\in \mathfrak{N}}a_{ij}g(x_i-x_j)\\
  =& \frac{1}{2}\sum_{i,j\in \mathfrak{N}}a_{ij}g(x_i-x_j)+\frac{1}{2}\sum_{i,j\in \mathfrak{N}}a_{ji}g(x_j-x_i)\\
  =& \frac{1}{2}\sum_{i,j\in \mathfrak{N}}a_{ij}g(x_i-x_j)-\frac{1}{2}\sum_{i,j\in \mathfrak{N}}a_{ij}g(x_i-x_j)
  =0
\end{align*}
\end{proof}
Then we propose the following estimator:
\begin{align}
z_{it}=&\mbox{ }\zeta_{it}+H_i(t,x_{it})\label{z_i},\\
\dot{\zeta}_{it}=&\sum_{j\in \mathfrak{N}}\tilde{a}_{ij}(-\alpha_1 sig^{p}(\hat{z}_{ij}) -\beta_1sig^{q}(\hat{z}_{ij})-\gamma_2sign(\hat{z}_{ij})),\label{zeta}
\end{align}
where $\alpha_1,\beta_1,\gamma_2\in\mathbb{R}_
  {++}$, $\hat{z}_{ij} =z_{it}-z_{it}$. In turn, we can obtain the following theorem.
\begin{notation}\label{note}
 We denote that $\mathcal{L}_p$, $\mathcal{L}_q$, $\mathcal{L}_2$, $\mathcal{L}_1$ represent respectively the Laplacian matrix of $[\tilde{a}_{ij}^{\frac{2}{p+1}}]_{n\times n}$, $[\tilde{a}_{ij}^{\frac{2}{q+1}}]_{n\times n}$, $ [\tilde{a}_{ij}^2]_{n\times n}$, $[\tilde{a}_{ij}]_{n\times n}$. Moreover by Lemma \ref{lemma-eig}, we can obtain the following inequalities where $\mathbf{e}=[e_1^T,e_2^T,\cdots,e_n^T]^T$ ($e_i\in\mathbb{R}^n,\,\forall i\in \mathfrak{N}$):
\begin{align*}
&\sum_{i,j\in \mathfrak{N}}\tilde{a}_{ij}^{\frac{2}{p+1}}\|e_i-e_j\|^2 =2\mathbf{e}^T\left(\mathcal{L}_{p}\bigotimes I_m\right)\mathbf{e}\geq 2\lambda_{2}[\mathcal{L}_{p}]\mathbf{e}^T\mathbf{e} 
\\
&\sum_{i,j\in \mathfrak{N}}\tilde{a}_{ij}^{\frac{2}{q+1}}\|e_i-e_j\|^2=2\mathbf{e}^T\left(\mathcal{L}_{q}\bigotimes I_m\right)\mathbf{e}\geq 2\lambda_{2}[\mathcal{L}_{q}]\mathbf{e}^T\mathbf{e}
\\
&\sum_{i,j\in \mathfrak{N}}\tilde{a}_{ij}^{2}\|e_i-e_j\|^2=2\mathbf{e}^T\left(\mathcal{L}_{2}\bigotimes I_m\right)\mathbf{e}\geq 2\lambda_{2}[\mathcal{L}_{2}]\mathbf{e}^T\mathbf{e},\\
&\sum_{i,j\in \mathfrak{N}}\tilde{a}_{ij}\|e_i-e_j\|^2=2\mathbf{e}^T\left(\mathcal{L}_{1}\bigotimes I_m\right)\mathbf{e}\geq 2\lambda_{2}[\mathcal{L}_{1}]\mathbf{e}^T\mathbf{e},
\end{align*}
\end{notation}
\begin{theorem}\label{the-estimate}
Let the Assumptions \ref{A1}-\ref{A2} hold. Suppose that $\gamma_2\geq \frac{L_H\sqrt{2N}}{\sqrt{\lambda_2[\mathcal{L}_2]}}$ and the initial value of $\zeta_{it}$ be 0, i.e. $\sum_{i\in \mathfrak{N}}\zeta_{i0}=0$, \eqref{z_i}-\eqref{zeta} can enable $z_{it}$ converge to $\frac{1}{N}\sum_{j\in \mathfrak{N}}H_j(t,x_{jt}),\, \forall i\in \mathfrak{N}$ within a fixed-time $t\leq T_1$, where \[T_1=\frac{1}{2^p\alpha_1\lambda_2^{\frac{p+1}{2}}[\mathcal{L}_p](1-p)}+
\frac{ n^{\frac{q-1}{2}}N^{q-1}}{2^q\beta_1\lambda_2^{\frac{q+1}{2}}[\mathcal{L}_q](q-1)}.\]
\end{theorem}
\begin{proof}
 Using Lemma \ref{lemma-0} yields
 \begin{align*}
 &\sum_{i\in \mathfrak{N}}\dot{\zeta}_{it}(t)\\
  =&
 \sum_{i,j\in \mathfrak{N}}\tilde{a}_{ij}(-\alpha_1 sig^{p}(\hat{z}_{ij}) -\beta_1sig^{q}(\hat{z}_{ij})-\gamma_2sign(\hat{z}_{ij}))=0
 \end{align*}
 With $\sum_{i\in \mathfrak{N}}\zeta_{i0}=0$ and $\sum_{i\in \mathfrak{N}}\dot{\zeta}_{it} = 0$, we easily get  $\sum_{i\in \mathfrak{N}}\zeta_{it}\equiv0$ for all $t\geq 0$. Furthermore, we have $\sum_{i\in \mathfrak{N}}z_{it}=\sum_{i\in \mathfrak{N}}\zeta_{it}+\sum_{i\in \mathfrak{N}}H_i(t,x_{it})=\sum_{i\in \mathfrak{N}}H_i(t,x_{it})$.

Let $ \hat{z}_i = z_{it}-\frac{1}{N}\sum_{j\in \mathfrak{N}}z_{jt}$, so that
$\sum_{i\in \mathfrak{N}} \hat{z}_i = 0$ and
 $\hat{z}_i-\hat{z}_j=\hat{z}_{ij} =z_{it}-z_{it}$.
Now, consider a positive definite function $V_2=\frac{1}{2}\sum_{i\in \mathfrak{N}}\hat{z}_i^T\hat{z}_i$, of which the derivative of $t$ is given by
\begin{align*}
\dot{V}_2
=& \sum_{i\in \mathfrak{N}}\hat{z}_i^T\dot{\hat{z}}_i
= \sum_{i\in \mathfrak{N}}\hat{z}_i^T\dot{z}_{it}
-\left(\sum_{i\in \mathfrak{N}}\hat{z}_i^T\right)\left(\frac{1}{N}\sum_{j\in \mathfrak{N}}\dot{z}_{jt}\right) \\
=& \sum_{i\in \mathfrak{N}}\hat{z}_i^T\dot{z}_{it}-\left(\sum_{i\in \mathfrak{N}}\hat{z}_i^T\right)\left(\frac{1}{N}\sum_{j\in \mathfrak{N}}\frac{d}{d t}H_j(t,x_{jt})\right)\\
=& \underbrace{-\alpha_1\sum_{i,j\in \mathfrak{N}}\tilde{a}_{ij}\hat{z}_i^T sig^p(\hat{z}_{ij})}_{W_3}
\underbrace{-\beta_1\sum_{i,j\in \mathfrak{N}}\tilde{a}_{ij}\hat{z}_i^T sig^q(\hat{z}_{ij})}_{W_4}
\\
&\underbrace{- \gamma_2\sum_{i,j\in \mathfrak{N}}\tilde{a}_{ij}\hat{z}_i^T sign(\hat{z}_{ij})}_{W_5}\\
&\underbrace{-\frac{1}{N}\sum_{i,j\in \mathfrak{N}}\hat{z}_i^T \left(\frac{d}{d t}H_i(t,x_{it})-\frac{d}{d t}H_j(t,x_{jt})\right)}_{W_6}
\end{align*}
Applying the inequalities in Notation \ref{note}, we can perform the following operations:
\begin{align}
W_3=&-\alpha_1\sum_{i,j\in \mathfrak{N}}\tilde{a}_{ij}\hat{z}_i^T sig^p(\hat{z}_{ij})
=-\frac{\alpha_1}{2}\sum_{i,j\in \mathfrak{N}}\tilde{a}_{ij}\|\hat{z}_{ij}\|_{p+1}^{p+1}\nonumber\\
\leq&-\frac{\alpha_1}{2}\sum_{i,j\in \mathfrak{N}}\tilde{a}_{ij}\|\hat{z}_{ij}\|^{p+1}
=-\frac{\alpha_1}{2}\sum_{i,j\in \mathfrak{N}}\left(\tilde{a}_{ij}^{\frac{2}{p+1}}\|\hat{z}_{ij}\|^2\right)^{\frac{p+1}{2}}\nonumber\\
\leq&-\frac{\alpha_1}{2}\left(\sum_{i,j\in \mathfrak{N}}\tilde{a}_{ij}^{\frac{2}{p+1}}\|\hat{z}_{ij}\|^2\right)^{\frac{p+1}{2}}
\leq-\frac{\alpha_1}{2}(4\lambda_2[\mathcal{L}_p]V)^{\frac{p+1}{2}}\nonumber\\
\leq&-2^p\alpha_1\lambda_2^{\frac{p+1}{2}}[\mathcal{L}_p]V_2^{\frac{p+1}{2}};\label{W3}\\
W_4=&-\beta_1\sum_{i,j\in \mathfrak{N}}\tilde{a}_{ij}\hat{z}_i^T sig^q(\hat{z}_{ij})
=-\frac{\beta_1}{2}\sum_{i,j\in \mathfrak{N}}\tilde{a}_{ij}\|\hat{z}_{ij}\|_{q+1}^{q+1}\nonumber\\
\leq&-\frac{\beta_1}{2}n^{\frac{1-q}{2}}\sum_{i,j\in \mathfrak{N}}\tilde{a}_{ij}\|\hat{z}_{ij}\|^{q+1}\nonumber\\
=&-\frac{\beta_1}{2}n^{\frac{1-q}{2}}\sum_{i,j\in \mathfrak{N}}\left(\tilde{a}_{ij}^{\frac{2}{q+1}}\|\hat{z}_{ij}\|^2\right)^{\frac{q+1}{2}}\nonumber\\
\leq&-\frac{\beta_1}{2}n^{\frac{1-q}{2}}N^{1-q}\left(\sum_{i,j\in \mathfrak{N}}\tilde{a}_{ij}^{\frac{2}{q+1}}\|\hat{z}_{ij}\|^2\right)^{\frac{q+1}{2}}\nonumber\\
\leq&-\frac{\beta_1}{2}n^{\frac{1-q}{2}}N^{1-q}(4\lambda_2[\mathcal{L}_q]V_2)^{\frac{q+1}{2}}\nonumber\\
\leq&-2^q\beta_1n^{\frac{1-q}{2}}N^{1-q}\lambda_2^{\frac{q+1}{2}}[\mathcal{L}_q]V_2^{\frac{q+1}{2}};\label{W4}\\
W_5=&- \gamma_2\sum_{i,j\in \mathfrak{N}}\tilde{a}_{ij}\hat{z}_i^T sign(\hat{z}_{ij})
 =- \frac{\gamma_2}{2}\sum_{i,j\in \mathfrak{N}}\tilde{a}_{ij}\hat{z}_{ij}^T sign(\hat{z}_{ij})\nonumber\\
  =&- \frac{\gamma_2}{2}\sum_{i,j\in \mathfrak{N}}\tilde{a}_{ij}\|\hat{z}_{ij}\|_1
  \leq- \frac{\gamma_2}{2}\sum_{i,j\in \mathfrak{N}}\tilde{a}_{ij}\|\hat{z}_{ij}\|\nonumber\\
  =&- \frac{\gamma_2}{2}\sum_{i,j\in \mathfrak{N}}(\tilde{a}_{ij}^2\|\hat{z}_{ij}\|^2)^{\frac{1}{2}}
  \leq- \frac{\gamma_2}{2}\left(\sum_{i,j\in \mathfrak{N}}\tilde{a}_{ij}^2\|\hat{z}_{ij}\|^2\right)^{\frac{1}{2}}\nonumber\\
  \leq&- \frac{\gamma_2}{2}\left(\sum_{i,j\in \mathfrak{N}}\tilde{a}_{ij}^2\|\hat{z}_{ij}\|^2\right)^{\frac{1}{2}}
  \leq- \frac{\gamma_2}{2}(2\lambda_2[\mathcal{L}_2]\sum_{i\in \mathfrak{N}}\|\hat{z}_{i}\|^2)^{\frac{1}{2}}\nonumber\\
  \leq&- \frac{\gamma_2\sqrt{\lambda_2[\mathcal{L}_2]}}{\sqrt{2}}(\sum_{i\in \mathfrak{N}}\|\hat{z}_{i}\|^2)^{\frac{1}{2}}
\leq- \frac{\gamma_2\sqrt{\lambda_2[\mathcal{L}_2]}}{\sqrt{2N}}\sum_{i\in \mathfrak{N}}\|\hat{z}_{i}\|.\label{W5}
\end{align}
For $W_6$, according to Assumption \ref{A2}, we have
\begin{align}\label{W6}
W_6 =& -\frac{1}{N}\sum_{i,j\in \mathfrak{N}}\hat{z}_i^T \left(\frac{d}{d t}H_i(t,x_{it})-\frac{d}{d t}H_j(t,x_{jt})\right)\nonumber\\
=& -\frac{1}{2N}\sum_{i,j\in \mathfrak{N}}\hat{z}_{ij}^T \left(\frac{d}{d t}H_i(t,x_{it})-\frac{d}{d t}H_j(t,x_{jt})\right)\nonumber\\
\leq& \frac{1}{2N}\sum_{i,j\in \mathfrak{N}}\|\hat{z}_{ij}\| \cdot \left\|\frac{d}{d t}H_i(t,x_{it})-\frac{d}{d t}H_j(t,x_{jt})\right\|\nonumber\\
\leq& \frac{L_H}{2N}\sum_{i,j\in \mathfrak{N}}\|\hat{z}_{ij}\|
\leq \frac{L_H}{2N}\sum_{i,j\in \mathfrak{N}}\left(\|\hat{z}_{i}\|+ \|\hat{z}_{j}\|\right)\nonumber\\
\leq& L_H\sum_{i\in \mathfrak{N}}\|\hat{z}_{i}\|.
\end{align}
Due to $\gamma_2\geq \frac{L_H\sqrt{2N}}{\sqrt{\lambda_2[\mathcal{L}_2]}}$, we synthesize the conclusions of \eqref{W3}-\eqref{W6} to obtain
\[\dot{V}_2\leq-2^p\alpha_1\lambda_2^{\frac{p+1}{2}}[\mathcal{L}_p]V_2^{\frac{p+1}{2}}
-2^q\beta_1n^{\frac{1-q}{2}}N^{1-q}\lambda_2^{\frac{q+1}{2}}[\mathcal{L}_q]V_2^{\frac{q+1}{2}}.\]
With Lemma \ref{lemma-FXC}, we have that $\varepsilon_{i}$ converge to zero for all $i\in \mathfrak{N}$, i.e. $z_{it}$ converge to $\frac{1}{N}\sum_{j\in \mathfrak{N}}z_{jt} = \frac{1}{N}\sum_{j\in \mathfrak{N}}H_j(t,x_{jt})$, within a fixed-time $T_1$, where
\[T_1=\frac{1}{2^p\alpha_1\lambda_2^{\frac{p+1}{2}}[\mathcal{L}_p](1-p)}+
\frac{ n^{\frac{q-1}{2}}N^{q-1}}{2^q\beta_1\lambda_2^{\frac{q+1}{2}}[\mathcal{L}_q](q-1)}.\]
\end{proof}
From Theorem \ref{the-estimate}, it follows that our proposed estimator \eqref{z_i}-\eqref{zeta} can estimate the global variable $\frac{1}{N}\sum_{j\in \mathfrak{N}}H_j(t,x_{jt})$ within a fixed time $T_1$. Following from this, we can design a distributed protocol to solve Problem \ref{P2}. 

\begin{assumption}\label{A4}
There are constants $L_1,L_2, L_4\in\mathbb{R}_
  {++}$ and $L_3,L_5\in\mathbb{R}_
  {+}$ such that, the objective function $f_i(t,x)$ satisfies the following conditions:
\begin{enumerate}
\item{$f_i(t,x)$ is $L_1$-strongly convex on $x$, $\forall i \in \mathfrak{N}$;}
\item{for all $x,y\in \mathbb{R}^n$, $t\geq 0$, $i,j\in \mathfrak{N}$,  $\left\|\nabla f_i(t,x) - \nabla f_j(t,y)\right\|\leq L_2\|x-y\|+L_3$,
$\left\|\nabla_{xt} f_i(t,x) - \nabla_{xt} f_j(t,y)\right\|\leq L_4\|x-y\|+L_5$ ;}
\item{$H_i(t,x)$ is a positive definite matrix and satisfies $\lambda_{\min}(H_i(t,x))\geq h_d>0$, $\forall x\in \mathbb{R}^n , t\geq 0, i\in \mathfrak{N}$. Further we assume that $\sum_{i\in \mathfrak{N}}H_i(t,x)=\sum_{i\in \mathfrak{N}}H_i(t,y)$, $\forall x,y\in \mathbb{R}^{n}$.}
\end{enumerate}
\end{assumption}

\begin{remark}
 In Assumption \ref{A4} (i) and (ii), we assume that all $f_i(t,x)$ are strongly convex (ensuring the existence and uniqueness of optimal trajectory) and impose specific continuity criteria on $\nabla f_i(t,x)$ and $\nabla_{xt} f_i(t,x)$ (ensuring convergence). Instead of assuming $H_i(t,x)=H_j(t,y),\,\forall i,j\in \mathfrak{N}$ in \cite{2016-rahili,2017-ning}, the Assumption \ref{A4} (iii) requires that the global objective function satisfies $\sum_{i\in \mathfrak{N}}H_i(t,x)=\sum_{i\in \mathfrak{N}}H_i(t,y)$. Clearly, there are a large number of functions that satisfy the criteria of Assumption \ref{A4}, including quadratic functions, trigonometric functions, negative power exponential functions, and some quadratic continuously differentiable functions, etc.
 It is noteworthy that the assumptions about the objective function in \cite{2016-rahili,2017-sun, 2017-ning,2020-li,fix4} are much stronger and are a special case of assumption \ref{A4}.

\end{remark}
\begin{lemma}\label{lemma-H}
  With Assumption \ref{A4}, we can conclude that $\left\|\left(\frac{1}{N}\sum_{i\in \mathfrak{N}}H_i(t,x_{i})\right)^{-1}\right\|\leq \frac{1}{h_d}$.
\end{lemma}
\begin{proof}
Based on Assumption \ref{A4} and Lemma \ref{weyl}, we have
\begin{align*}
& \left\|\left(\frac{1}{N}\sum_{i\in \mathfrak{N}}H_i(t,x_i)\right)^{-1}\right\|
= N\left\|\left(\sum_{i\in \mathfrak{N}}H_i(t,x_i)\right)^{-1}\right\|\\
=& N\lambda_{\max}\left(\left(\sum_{i\in \mathfrak{N}}H_i(t,x_i)\right)^{-1}\right)
 =\frac{N}{\lambda_{min}(\sum_{i\in \mathfrak{N}}H_i(t,x_i))}\\
\leq& \frac{N}{\sum_{i\in \mathfrak{N}}\lambda_{min}(H_i(t,x_i))}
\leq  \frac{1}{h_d}.
\end{align*}
\end{proof}
Therefore we propose the following protocol based on a distributed approach:
\begin{align}
u_i^d=& \sum_{j\in \mathfrak{N}}\tilde{a}_{ij}\left[-\alpha_2 sig^{p}(\varepsilon_{ij})-\beta_2sig^{q}(\varepsilon_{ij})-\gamma_3sig(\varepsilon_{ij})\right] \nonumber\\
&-\gamma_4 \nabla f_i(t,x_{it})- z_i^{-1}\nabla_{xt} f_i(t,x_{it}),
\label{u_d}
\end{align}
where $\alpha_2,\beta_2,\gamma_3,\gamma_4\in\mathbb{R}_
  {++}$, $\varepsilon_{ij} = x_{it}-x_{jt}$. Then we give the following theorems proving that all agents' states converge to the optimal trajectory under the protocol \eqref{u_d}.
\begin{theorem}\label{the-consensus}
Under Assumptions \ref{A1}-\ref{A4} and $\gamma_3> \frac{2\gamma_4L_2h_d +2L_4}{h_d\lambda_{2}[\mathcal{L}_{1}]}$, the protocol \eqref{u_d} enables the SMASs \eqref{x_i} to achieve practically fixed-time consensus (Pfxc) in probability and the stochastic settling time satisfies that $t\leq T_1+\tau$, where
\[\tau = \inf \left\{t:\mathbb{E}\left(\frac{1}{2}\sum_{i\in \mathfrak{N}}\left\|x_{it}-\frac{1}{N}\sum_{j\in \mathfrak{N}}x_{jt}\right\|^2\right)\leq \delta\right\}\]
 and
\[\mathbb{E}(\tau)\leq T_2 =\frac{2(m_1/m_2)^{\frac{1-p}{q-p}}}{m_1(1-p)}+ \frac{2(m_1/m_2)^{\frac{1-q}{q-p}}}{m_2(q-1)},\]
with $m_1 = k_1-k_3\delta^{-\frac{p+1}{2}}>0$,
$m_2 = k_2-k_3\delta^{-\frac{q+1}{2}}>0$,
$\delta = \min \left\{\left(\frac{k_3}{(1-\theta)k_1}\right)^{\frac{2}{p+1}},
\left(\frac{k_3}{(1-\theta)k_2}\right)^{\frac{2}{q+1}}\right\}$,
$k_1 = 2^p\alpha_2\lambda_2^{\frac{p+1}{2}}[\mathcal{L}_p]$,
$k_2 = 2^q\beta_2n^{\frac{1-q}{2}}N^{1-q}\lambda_2^{\frac{q+1}{2}}[\mathcal{L}_q]$,
$k_3=\frac{2\gamma_4L_3^2}{L_2}+\frac{2L_5^2}{h_dL_4}+\hat{\sigma}^2$, and $0<\theta<1$ is a design constant.
\end{theorem}
\begin{proof}
It follows from Theorem \ref{the-estimate} that $z_{it}=\frac{1}{N}\sum_{j\in \mathfrak{N}}H_j(t,x_{jt})$ for all $t\geq T_1, \, i\in \mathfrak{N}$. Let $\bar{x}_{t} =\frac{1}{N}\sum_{j\in \mathfrak{N}}x_{jt}$. Then by the SMASs \eqref{x_i}, we have
\begin{align*}
 &d\bar{x}_{t}
 =\frac{1}{N}\sum_{j\in \mathfrak{N}} dx_{jt}\\
=&\left[\frac{1}{N}\sum_{j,k\in \mathfrak{N}}\tilde{a}_{jk} \left(-\alpha_2sig^{p}(\varepsilon_{jk})
-\beta_2sig^{q}(\varepsilon_{jk}) -\gamma_3sig(\varepsilon_{jk})\right)\right.\\
&\left.-\frac{\gamma_4}{N}\sum_{j\in \mathfrak{N}} \nabla f_j(t,x_{jt})- \frac{1}{N}\sum_{j\in \mathfrak{N}}z_{jt}^{-1}\nabla_{xt} f_j(t,x_{jt})\right]dt\\
&+\frac{1}{N}\sum_{j\in \mathfrak{N}}\sigma_j dB(t).
\end{align*}
Then we can easily obtained the following SDE by Lemma \ref{lemma-0},
\begin{align*}
 d\bar{x}_t
 =&\left(-\frac{\gamma_4}{N}\sum_{j\in \mathfrak{N}} \nabla f_j(t,x_{jt})- \frac{1}{N}\sum_{j\in \mathfrak{N}}z_{jt}^{-1}\nabla_{xt} f_j(t,x_{jt})\right)dt\\
&+\frac{1}{N}\sum_{j\in \mathfrak{N}}\sigma_j dB(t).
\end{align*}
Denote $\varepsilon_i = x_{it} - \bar{x}_t$. Then it is not difficult to see that $\sum_{i\in \mathfrak{N}}\varepsilon_i=0$ and $ \varepsilon_i-\varepsilon_j = \varepsilon_{ij} = x_{it}-x_{jt}$. Thus, we can obtain the following error system:
\begin{align*}
 &d\tilde{\varepsilon}_i\\
=& dx_{it} - d\bar{x}_t\\
=& \left[\frac{1}{N}\sum_{j,k\in \mathfrak{N}}\tilde{a}_{ij} \left(-\alpha_2sig^{p}(\varepsilon_{ij})
-\beta_2sig^{q}(\varepsilon_{ij}) -\gamma_3sig(\varepsilon_{ij})\right)\right.\\
&\left.-\gamma_4 \nabla f_i(t,x_{it})- z_{it}^{-1}\nabla_{xt} f_i(t,x_{it})+\frac{\gamma_4}{N}\sum_{j\in \mathfrak{N}} \nabla f_j(t,x_{jt})\right.\\
&\left.+\frac{1}{N}\sum_{j\in \mathfrak{N}}z_{jt}^{-1}\nabla_{xt} f_j(t,x_{jt})\right]dt+\tilde{\sigma}_idB(t).
\end{align*}
where $\tilde{\sigma}_i=\sigma_i -\frac{1}{N}\sum_{j\in \mathfrak{N}}\sigma_j$. For a positive definite function $V_3 = \frac{1}{2}\sum_{i\in \mathfrak{N}}\varepsilon_i^T\varepsilon_i$, by Definition \ref{def-LV}, we have
\begin{align*}
\mathcal{L}V_3
=&
\underbrace{-\alpha_2\sum_{i,j\in \mathfrak{N}}\tilde{a}_{ij}\varepsilon_i^T sig^p(\varepsilon_{ij})}_{W_7}
\underbrace{-\beta_2\sum_{i,j\in \mathfrak{N}}\tilde{a}_{ij}\varepsilon_i^T sig^q(\varepsilon_{ij})}_{W_8}\\
&\underbrace{-\gamma_3\sum_{i,j\in \mathfrak{N}}\tilde{a}_{ij}\varepsilon_i^T sig(\varepsilon_{ij})}_{W_9}\\
&\underbrace{- \frac{\gamma_4}{N}\sum_{i,j\in \mathfrak{N}}\varepsilon_i^T (\nabla f_i(t,x_{it})-\nabla f_j(t,x_{jt}))}_{W_{10}}\\
&\underbrace{-\frac{1}{N}\sum_{i,j\in \mathfrak{N}}\varepsilon_i^T z_{it}^{-1} (\nabla_{xt} f_i(t,x_{it}) -\nabla_{xt} f_j(t,x_{jt}))}_{W_{11}}\\
&+\underbrace{\frac{1}{2}\sum_{i\in \mathfrak{N}} trace[\tilde{\sigma}_i^T\tilde{\sigma}_i]}_{W_{12}}.
\end{align*}
Similar to \eqref{W3}-\eqref{W5}, using the inequalities in Notation \ref{note}, one has
\begin{align}
W_7=&-\alpha_2\sum_{i,j\in \mathfrak{N}}\tilde{a}_{ij}\varepsilon_i^T sig^p(\varepsilon_{ij})
\leq-2^p\alpha_2\lambda_2^{\frac{p+1}{2}}[\mathcal{L}_p]V_3^{\frac{p+1}{2}}\label{W7}\\
W_8=&-\beta_2\sum_{i,j\in \mathfrak{N}}\tilde{a}_{ij}\varepsilon_i^T sig^q(\varepsilon_{ij})\nonumber\\
\leq&-2^q\beta_2n^{\frac{1-q}{2}}N^{1-q}\lambda_2^{\frac{q+1}{2}}[\mathcal{L}_q]V_3^{\frac{q+1}{2}}\label{W8}\\
W_9=&-\gamma_3\sum_{i,j\in \mathfrak{N}}\tilde{a}_{ij}\varepsilon_i^T sig(\varepsilon_{ij})
=-\frac{\gamma_3}{2}\sum_{i,j\in \mathfrak{N}}\tilde{a}_{ij}\varepsilon_{ij}^T sig(\varepsilon_{ij})\nonumber\\
=&-\frac{\gamma_3}{2}\sum_{i,j\in \mathfrak{N}}\tilde{a}_{ij}\|\varepsilon_{ij}\|^2
\leq-\gamma_3\lambda_2[\mathcal{L}_1]\sum_{i\in \mathfrak{N}}\|\varepsilon_{i}\|^2\nonumber\\
=&-2\gamma_3\lambda_2[\mathcal{L}_1]V_3\label{W9}
\end{align}
Let $k_1 = 2^p\alpha_2\lambda_2^{\frac{p+1}{2}}[\mathcal{L}_p]$ and $k_2 = 2^q\beta_2n^{\frac{1-q}{2}}N^{1-q}\lambda_2^{\frac{q+1}{2}}[\mathcal{L}_q]$.
By Assumption \ref{A4} and Lemma \ref{lemma-H}, $W_{10}$, $W_{11}$ can be calculated as follows
\begin{align}
W_{10} =& - \frac{\gamma_4}{N}\sum_{i,j\in \mathfrak{N}}\varepsilon_i^T (\nabla f_i(t,x_{it})-\nabla f_j(t,x_{jt}))\nonumber\\
=& - \frac{\gamma_4}{2N}\sum_{i,j\in \mathfrak{N}}\varepsilon_{ij}^T (\nabla f_i(t,x_{it})-\nabla f_j(t,x_{jt}))\nonumber\\
\leq& \frac{\gamma_4}{2N}\sum_{i,j\in \mathfrak{N}}\|\varepsilon_{ij}^T\|\cdot \left\|\nabla f_i(t,x_{it})-\nabla f_j(t,x_{jt})\right\|\nonumber\\
\leq& \frac{\gamma_4}{2N}\sum_{i,j\in \mathfrak{N}}\|\varepsilon_{ij}\|(L_2\|\tilde{\varepsilon}_{ij}\|+L_3)\nonumber\\
\leq&\frac{\gamma_4L_2}{2N}\sum_{i,j\in \mathfrak{N}}\left(\|\varepsilon_{ij}\|+\frac{L_3}{L_2}\right)^2\nonumber\\
\leq& \frac{\gamma_4L_2}{N}\sum_{i,j\in \mathfrak{N}}\|\varepsilon_{ij}\|^2+\frac{2\gamma_4L_3^2}{L_2}\nonumber\\
=&2\gamma_4 L_1 \sum_{i\in \mathfrak{N}}\|\varepsilon_{i}\|^2+\frac{2\gamma_4L_3^2}{L_2}\nonumber\\
=& 4\gamma_4 L_1V_3+\frac{2\gamma_4L_3^2}{L_2},\label{W10}\\
W_{11} =& -\frac{1}{N}\sum_{i,j\in \mathfrak{N}}\varepsilon_i^T z_{it}^{-1} \left(\nabla_{xt} f_i(t,x_{it}) -\nabla_{xt}f_j(t,x_{jt})\right)\nonumber\\
=& -\frac{1}{2N}\sum_{i,j\in \mathfrak{N}}\varepsilon_{ij}^T z_{it}^{-1} \left(\nabla_{xt}f_i(t,x_{it}) -\nabla_{xt}f_j(t,x_{jt})\right)\nonumber\\
\leq& \frac{1}{2N} \sum_{i,j\in \mathfrak{N}}\|\varepsilon_{ij}\| \cdot \|z_{it}^{-1}\|\cdot \left\|\nabla_{xt} f_i(t,x_{it}) -\nabla_{xt} f_j(t,x_{jt})\right\|\nonumber\\
\leq& \frac{1}{2Nh_d}\sum_{i,j\in \mathfrak{N}}\|\varepsilon_{ij}\|\left(L_4\|\varepsilon_{ij}\|+L_5\right)\nonumber\\
\leq& \frac{L_4}{2Nh_d}\sum_{i,j\in \mathfrak{N}}\left(\|\varepsilon_{ij}\|+\frac{L_5}{L_4}\right)^2\nonumber\\
\leq& \frac{2L_4}{2Nh_d}\sum_{i,j\in \mathfrak{N}}\|\varepsilon_{ij}\|^2+\frac{2L_5^2}{h_dL_4}\nonumber\\
\leq& \frac{2L_4}{h_d}\sum_{i\in \mathfrak{N}}\|\varepsilon_{i}\|^2+\frac{2L_5^2}{h_dL_4}
\leq \frac{4L_4}{h_d}V_3+\frac{2L_5^2}{h_dL_4}.\label{W11}
\end{align}
With \eqref{W9}-\eqref{W11} and $\gamma_3> \frac{2\gamma_4L_2h_d +2L_4}{h_d\lambda_{2}[\mathcal{L}_{1}]}$, we have
$W_9+W_{10}+W_{11}\leq \frac{2\gamma_4L_3^2}{L_2}+\frac{2L_5^2}{h_dL_4}.$
As for $W_{12}$, it follows from Assumption \ref{C-A1} that
\begin{align}\label{W12}
W_{12}=&\frac{1}{2}\sum_{i\in \mathfrak{N}} trace[\tilde{\sigma}_i^T\tilde{\sigma}_i]
=\frac{1}{2}\sum_{i\in \mathfrak{N}}\|\tilde{\sigma}_i\|_F^2\nonumber\\
\leq& \frac{1}{2N^2}\sum_{i,j\in \mathfrak{N}}\|\sigma_i -\sigma_j\|_F^2
 = \bar{\sigma}^2.
\end{align}
Then combining the results of \eqref{W7}-\eqref{W12} gives
\[
\mathcal{L}V_3
\leq-k_1V_3^{\frac{p+1}{2}}-k_2V_3^{\frac{q+1}{2}}+k_3,
\]
where $k_1 = 2^p\alpha_2\lambda_2^{\frac{p+1}{2}}[\mathcal{L}_p]$, $k_2 = 2^q\beta_2n^{\frac{1-q}{2}}N^{1-q}\lambda_2^{\frac{q+1}{2}}[\mathcal{L}_q]$, and $k_3 = \frac{2\gamma_4L_3^2}{L_2}+\frac{2L_5^2}{h_dL_4}+\bar{\sigma}^2$.
With Lemma \ref{lemma-Sfxs}, we have that the solution of the SMASs \eqref{x_i} is Pfxc in probability and the stochastic settling time satisfies $t\leq T_1+\tau$, where
$\tau = \inf \left\{t:\mathbb{E}(V_3)\leq \delta\right\}$
 and
\[\mathbb{E}(\tau)\leq T_2 =\frac{2(m_1/m_2)^{\frac{1-p}{q-p}}}{m_1(1-p)}+ \frac{2(m_1/m_2)^{\frac{1-q}{q-p}}}{m_2(q-1)},\]
with $m_1 = k_1-k_3\delta^{-\frac{p+1}{2}}>0$,
$m_2 = k_2-k_3\delta^{-\frac{q+1}{2}}>0$,
$\delta = \min \left\{\left(\frac{k_3}{(1-\theta)k_1}\right)^{\frac{2}{p+1}},
\left(\frac{k_3}{(1-\theta)k_2}\right)^{\frac{2}{q+1}}\right\}$,
$k_1 = 2^p\alpha_2\lambda_2^{\frac{p+1}{2}}[\mathcal{L}_p]$,
$k_2 = 2^q\beta_2n^{\frac{1-q}{2}}N^{1-q}\lambda_2^{\frac{q+1}{2}}[\mathcal{L}_q]$,
$k_3=\frac{2\gamma_4L_3^2}{L_2}+\frac{2L_5^2}{h_dL_4}+\bar{\sigma}^2$, and $0<\theta<1$ is a design constant.
\end{proof}
Based on the conclusion of Theorem \ref{the-consensus} that all agents achieve Pfxc in probability, further we can get the following theorem.
\begin{theorem}\label{the-optimal}
  Under Assumptions \ref{A1}-\ref{A4} and $\gamma_4 > \frac{4L_4}{h_d L_1}$, the tracking error of the agent with respect to the optimal trajectory $x^*_t$ is MS-GEUB given protocol \eqref{u_d}, i.e.
  \[\mathbb{E}(\frac{1}{N}\sum_{i\in \mathfrak{N}}\|x_{it}-x_t^*\|^2)\leq \mathbb{E}(\frac{1}{N}\sum_{i\in \mathfrak{N}}\|x_{i0}-x_0^*\|^2)\exp\{-k_5t\}+\frac{2k_5}{k_4}\]
  for all $t>0$, where $k_4 =L_1\gamma_4 -4L_4h_d^{-1} >0$, $k_5 = \frac{2L^2_5}{L_4h_d}+\frac{\bar{\sigma}^2}{2}$.
\end{theorem}

\begin{proof}
Following from the results presented in Theorem \ref{the-estimate}-\ref{the-consensus}, $z_{it}=\frac{1}{N}\sum_{j\in \mathfrak{N}}H_j(t,x_{jt})$ for all $t\geq T_1$ and all agents achieve Pfxc in probability with the settling time $t\leq T_1+\tau$, where $\mathbb{E}(\tau)\leq T_2$. Then the protocol can be rewritten as
\[u_i^d = -\gamma_4 \nabla f_i(t,x_{it})- z_{it}^{-1}\nabla_{xt} f_i(t,x_{it}).\]
Following from Lemma \ref{lemma-min}, there exists an optimal solution to the objective function that satisfies the condition
\[x_t^*=\arg\min\left(\sum_{k\in \mathfrak{N}}f_k(t,x_i)\right)
   \Leftrightarrow\sum_{k\in \mathfrak{N}}\nabla f_k(t,x_t^*)=0.\]
Differentiating both sides of the equation on $t$ yields
\[dx_t^* = -\left(\sum_{k\in \mathfrak{N}}H_k(t,x_t^*)\right)^{-1}\left(\sum_{k\in \mathfrak{N}} \nabla_{xt} f_k(t,x_t^*)\right)dt.\]
 Based on the above system, we construct the following error system where $\hat{x}_{it} = x_{it}-x^*_t$,
\begin{align*}
    d\hat{x}_{it} =& \left[u_i^d+\left(\sum_{k\in \mathfrak{N}}H_k(t,x_t^*)\right)^{-1}\left(\frac{1}{N}\sum_{k\in \mathfrak{N}} \nabla_{xt} f_k(t,x_t^*)\right)\right]dt\\
    &+\sigma_idB(t).
\end{align*}
  Based on the above discussion, we consider a positive definite function $V_4 = \frac{1}{2}\sum_{i\in \mathfrak{N}}\hat{x}_{it}^T\hat{x}_{it}$. Due to Assumption \ref{A4}, we have $z_{it}=\sum_{k\in \mathfrak{N}}H_k(t,x_t^*)$, $\forall i\in \mathfrak{N}$, $t\geq T_1$. Then we have
\begin{align*}
& \mathcal{L}V_4\\
   =& \sum_{i\in \mathfrak{N}}\hat{x}_{it}^T\left(-\gamma_4 \nabla f_i(t,x_{it})-z_{it}^{-1}\nabla_{xt} f_i(t,x_{it})\right.\\
  &\left.+\left(\sum_{k\in \mathfrak{N}}H_k(t,x_t^*)\right)^{-1}\left(\frac{1}{N}\sum_{k\in \mathfrak{N}} \nabla_{xt} f_k(t,x_t^*)\right)\right)\\
  &+\frac{1}{2}\sum_{i\in \mathfrak{N}}trace[\sigma_i^T\sigma_i]\nonumber\\
  \leq&\frac{N\bar{\sigma}^2}{2}\underbrace{-\gamma_4\sum_{i\in \mathfrak{N}}\hat{x}_{it}^T\nabla f_i(t,x_{it})}_{W_{13}}\nonumber\\
  &\underbrace{-\sum_{i\in \mathfrak{N}}\hat{x}_i^T z_{it}^{-1}\left( \nabla_{xt} f_i(t,x_{it})-\frac{1}{N}\sum_{k\in \mathfrak{N}} \nabla_{xt} f_k(t,x_t^*)\right)}_{W_{14}}.
  \end{align*}
  By strong convexity of $f_i(t,x)$, one gets
  \begin{align}\label{W13}
    W_{13} =& -\gamma_4\sum_{i\in \mathfrak{N}}\hat{x}_{it}^T\nabla f_i(t,x_{it})\nonumber\\
    \leq&\gamma_4\left(\sum_{i\in \mathfrak{N}}f_i(t,x_{it})-\sum_{i\in \mathfrak{N}}f_i(t,x_t^*)\right)\nonumber\\
    &-\gamma_5\sum_{i\in \mathfrak{N}}\hat{x}_{it}^T\nabla f_i(t,x_{it})\nonumber\\
    =&\gamma_4\sum_{i\in \mathfrak{N}}(f_i(t,x_{it})-f_i(t,x_t^*)-\hat{x}_{it}^T\nabla f_i(t,x_{it}))\nonumber\\
    \leq& -\frac{L_1\gamma_4}{2}\sum_{i\in \mathfrak{N}}\|\hat{x}_{it}\|^2=-L_1\gamma_4V_4.
  \end{align}
By Assumption \ref{A4}, Lemmas \ref{lemma-inequality} and \ref{lemma-H}, $W_{14}$ can be calculated as
  \begin{align}\label{W14}
  & W_{14}\nonumber\\
    =&-\sum_{i\in \mathfrak{N}}\hat{x}_{it}^T z_{it}^{-1}\left(\nabla_{xt} f_i(t,x_{it})-\frac{1}{N}\sum_{k\in \mathfrak{N}}\nabla_{xt} f_k(t,x_t^*)\right)\nonumber\\
    =&-\sum_{i\in \mathfrak{N}}\hat{x}_{it}^T z_{it}^{-1}\cdot\frac{1}{N}\sum_{k\in \mathfrak{N}} \left(\nabla_{xt} f_i(t,x_{it})-\nabla_{xt} f_k(t,x_t^*)\right)\nonumber\\
    \leq&\frac{1}{N}\sum_{i\in \mathfrak{N}}\|\hat{x}_{it}\|\cdot \|z_{it}^{-1}\|\cdot\sum_{k\in \mathfrak{N}} \left\| \nabla_{xt} f_i(t,x_{it})-\nabla_{xt} f_k(t,x_t^*)\right\|\nonumber\\
    \leq&\frac{1}{Nh_d}\sum_{i\in \mathfrak{N}}\|\hat{x}_{it}\|\left(\sum_{k\in \mathfrak{N}} (L_4\|\hat{x}_{it}\|+L_5)\right)\nonumber\\
    =&\frac{L_4}{h_d}\sum_{i\in \mathfrak{N}}\|\hat{x}_{it}\| \left(\|\hat{x}_{it}\|+\frac{L_5}{L_4}\right)
    \leq\frac{L_4}{h_d}\sum_{i\in \mathfrak{N}}\left(\|\hat{x}_{it}\|+\frac{L_5}{L_4}\right)^2\nonumber\\
    \leq&2L_4h_d^{-1}\sum_{i\in \mathfrak{N}}\|\hat{x}_i\|^2+\frac{2NL^2_5}{L_4h_d}\nonumber\\
    =&4L_4h_d^{-1}V_4+\frac{2NL^2_5}{L_4h_d}
  \end{align}
  Combining the results of \eqref{W13}-\eqref{W14} gives
  $\mathcal{L}V_4 \leq-(L_1\gamma_4 -4L_4h_d^{-1})V_4+N\left(\frac{2L^2_5}{L_4h_d}+\frac{\bar{\sigma}^2}{2}\right)$, where $\gamma_4 > \frac{4L_4}{h_d L_1}$. Then following from Lemma \ref{lemma-bound} we can conclude that all agents' states will converge to the nearby of the optimal trajectory and
  \[0\leq \mathbb{E}[V_4]\leq \mathbb{E}[V_4(x_0)]\exp\{-k_5t\}+\frac{Nk_5}{k_4},\]
 where $k_4 =L_1\gamma_4 -4L_4h_d^{-1} >0$, $k_5 = \frac{2L^2_5}{L_4h_d}+\frac{\bar{\sigma}^2}{2}$. Then from the above inequality we get that $E[V_4]$ is MS-GEUB, i.e. \[\mathbb{E}(\frac{1}{N}\sum_{i\in \mathfrak{N}}\|x_{it}-x_t^*\|^2)\leq \mathbb{E}(\frac{1}{N}\sum_{i\in \mathfrak{N}}\|x_{i0}-x_0^*\|^2)\exp\{-k_5t\}+\frac{2k_5}{k_4}\]
 for all $t>0$.
\end{proof}
 \begin{remark}
Theorem \ref{the-optimal} is a stochastic version of Theorem 3.13 in \cite{2016-rahili} and Theorem 2 in \cite{2017-ning}, in which the distributed TV-OPs on MASs is extended to the one on SMASs.
 \end{remark}
\section{Simulations}
\noindent This section will verify the findings above through numerical simulations. We first give the numerical simulation of the centralized protocol \eqref{u_c}.

\begin{example}\label{ex1}
  For $n=2$, let $x_0 = (-5,5)$ and $F(t,x) =\frac{\exp(-t) + 1}{2} (x_1 - \cos(\pi t))^2 + \frac{2\exp(-t) + 1}{2} (x_2 - \sin(\pi t))^2$. Then the centralized optimization problem is
\begin{equation}
\left\{
\begin{array}{l}
\min \mathbb{E}[F(t,x_t)]\nonumber\\
d x_t = u^c dt + \begin{pmatrix}
 0.5 \sin(\pi t)  & 0 \\
 0 & 0.5 \cos(\pi t)
 \end{pmatrix}
 dB(t)
\end{array}
\right.
\end{equation}
According to the requirement for $\gamma_1$ in Theorem \ref{the-cen}, we take $\gamma_1 = 0.7$, then the centralized protocol will be
\[u^c = -0.7\nabla F(t,x_t)-H_F(t,x_t)^{-1}\nabla_{xt} F(t,x_t).\]
The result is displayed in the Figure \ref{fig1}, where the red lines in the first two figures indicate the optimal trajectories $x_{1t}^* =\cos(\pi t)$ and $x_{2t}^* =\sin(\pi t)$, the red line in the last figure shows the upper bound of the mean square tracking error, which can be calculated as 0.625. Observation of the figures shows that the protocol \eqref{u_c} achieves the desired results.
\renewcommand{\dblfloatpagefraction}{.99}
\begin{figure*}[!t]
\centering
  	\includegraphics[width=5.5cm]{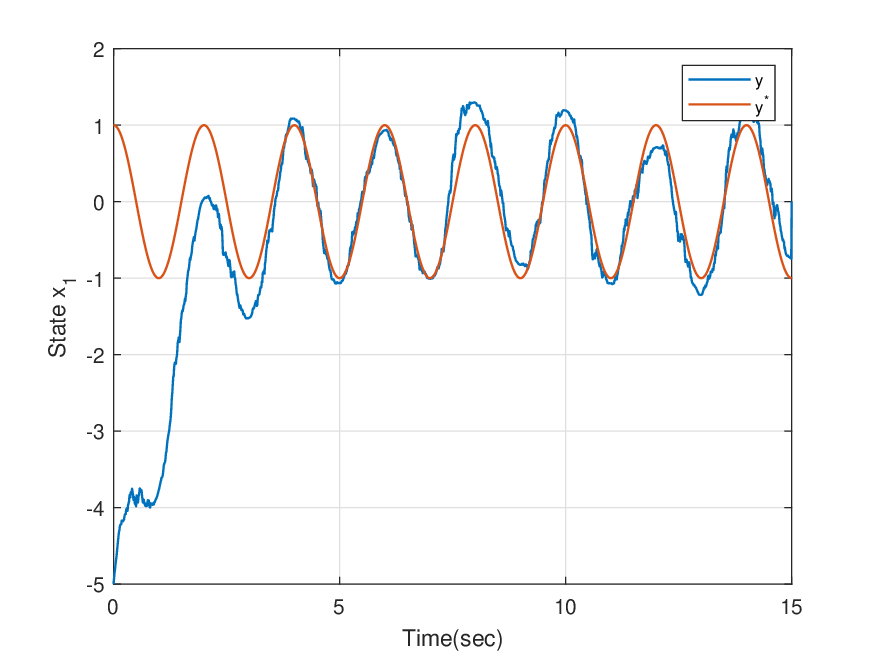}
  	\includegraphics[width=5.5cm]{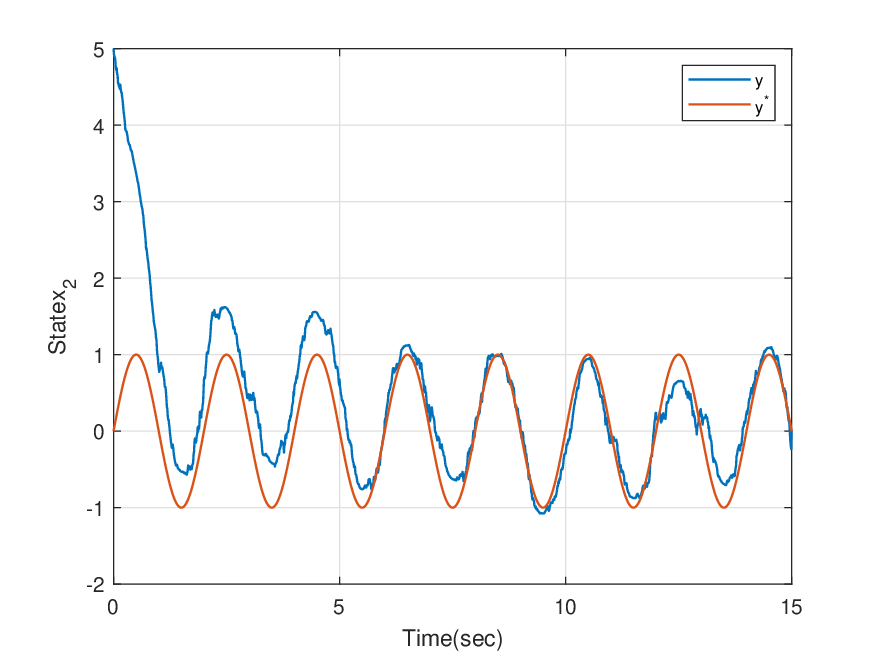}
    \includegraphics[width=5.5cm]{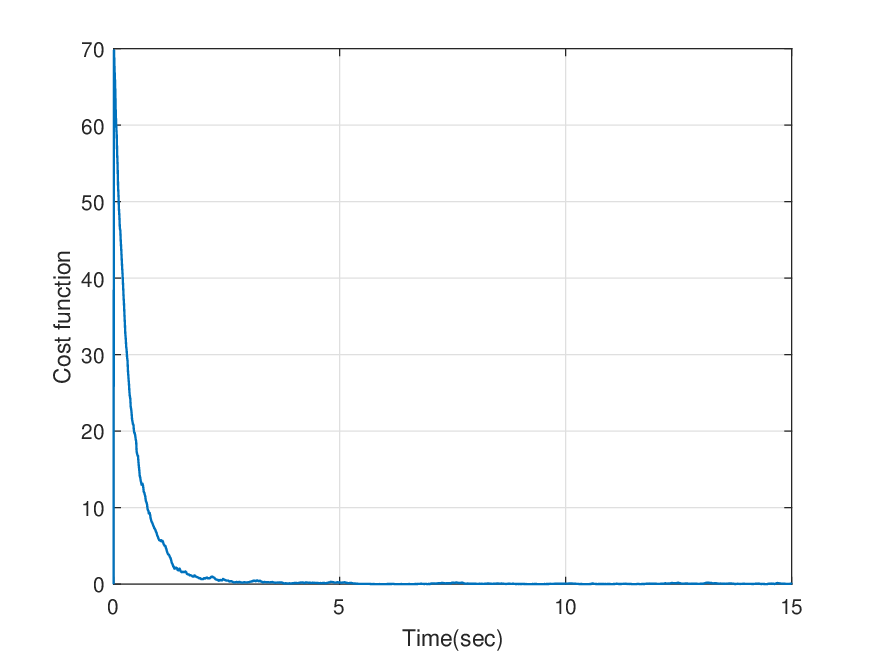}
  	\includegraphics[width=5.5cm]{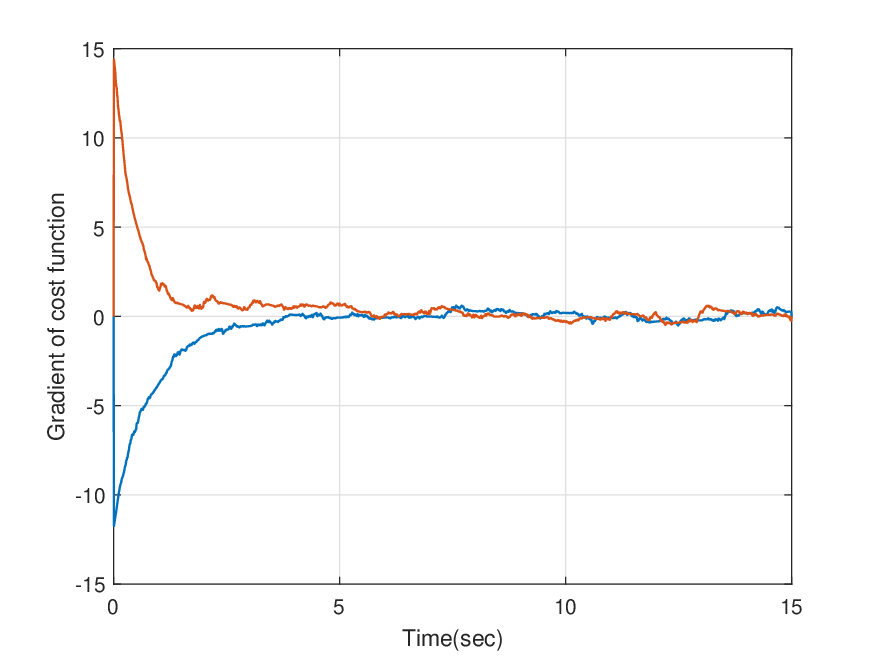}
	\includegraphics[width=5.5cm]{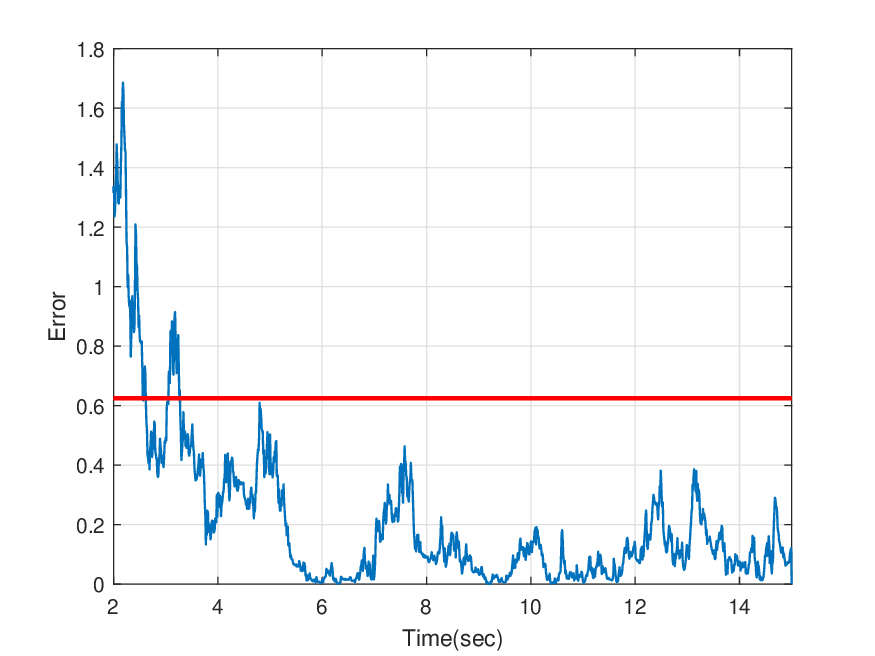}
\caption{\small The Results of Example \ref{ex1}.}
 \label{fig1}
\end{figure*}
\end{example}

Then we give the numerical simulation of the distributed protocol \eqref{u_d}.
\renewcommand{\dblfloatpagefraction}{.99}
\begin{figure*}[!t]
\centering
  	\includegraphics[width=5.5cm]{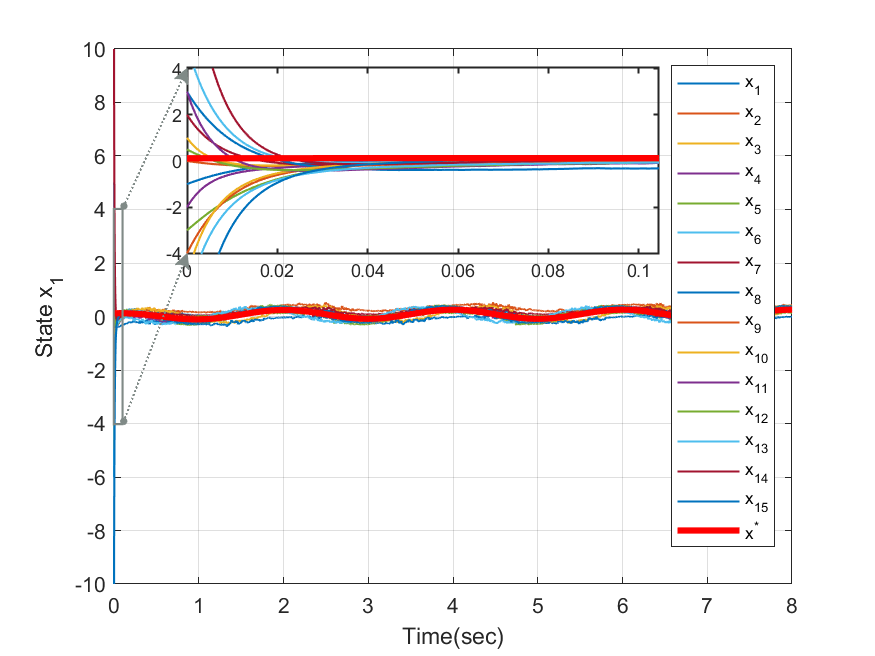}
  	\includegraphics[width=5.5cm]{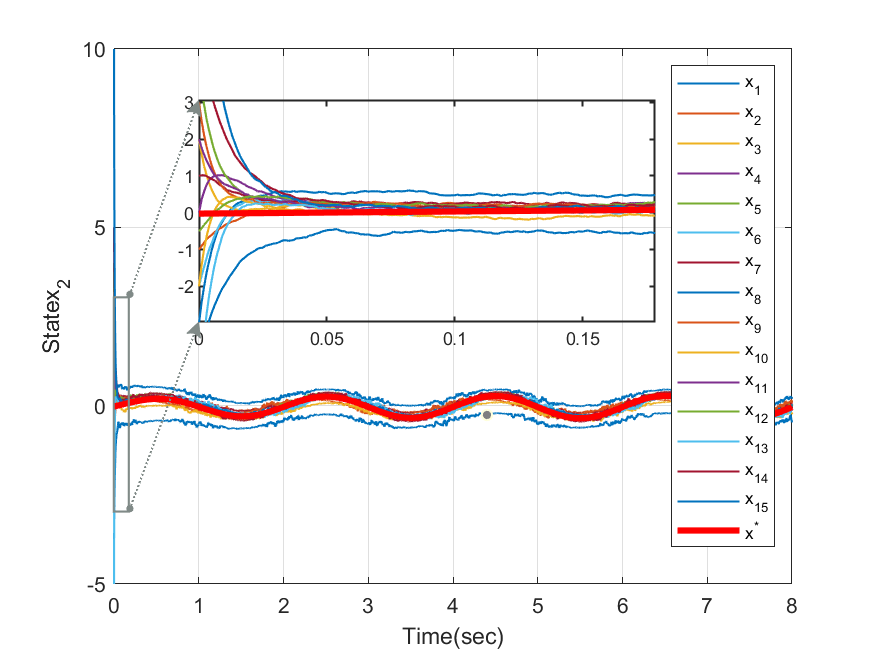}
    \includegraphics[width=5.5cm]{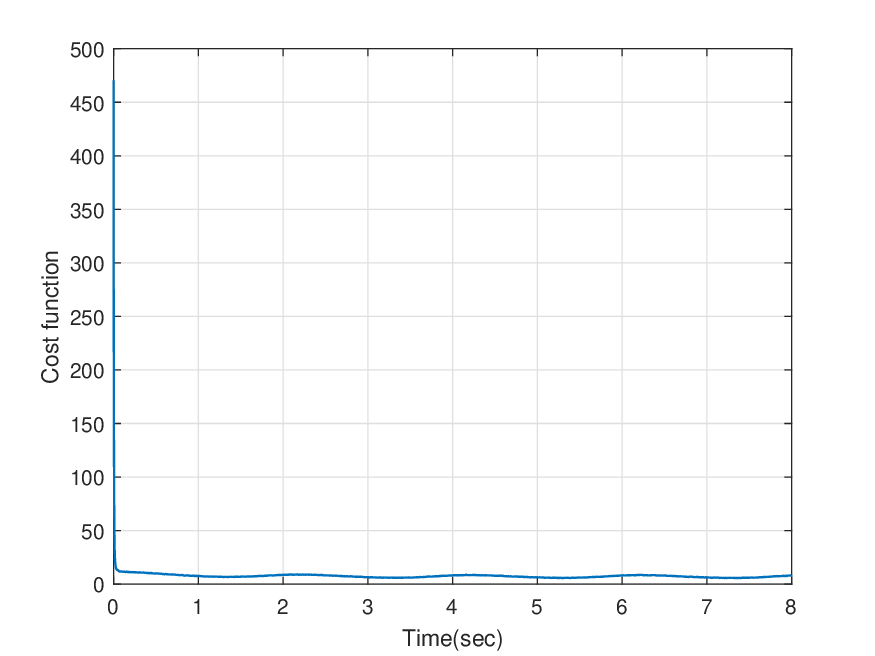}
    \includegraphics[width=5.5cm]{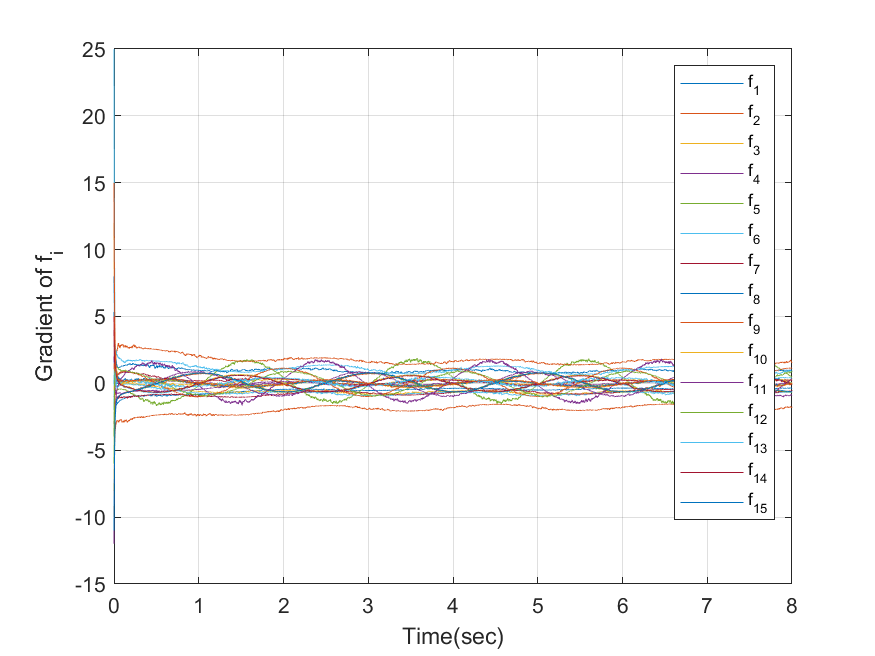}
  	\includegraphics[width=5.5cm]{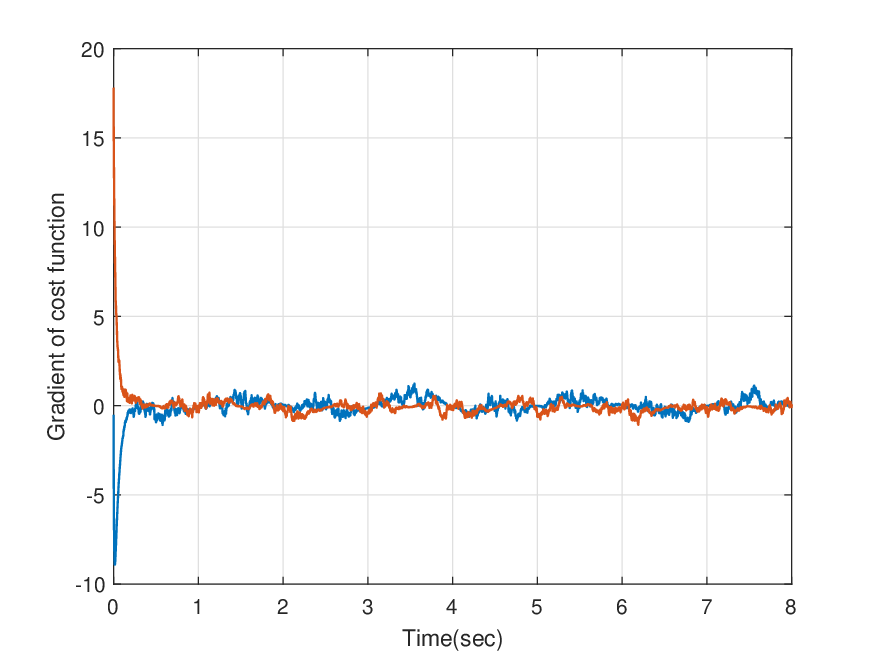}
	\includegraphics[width=5.5cm]{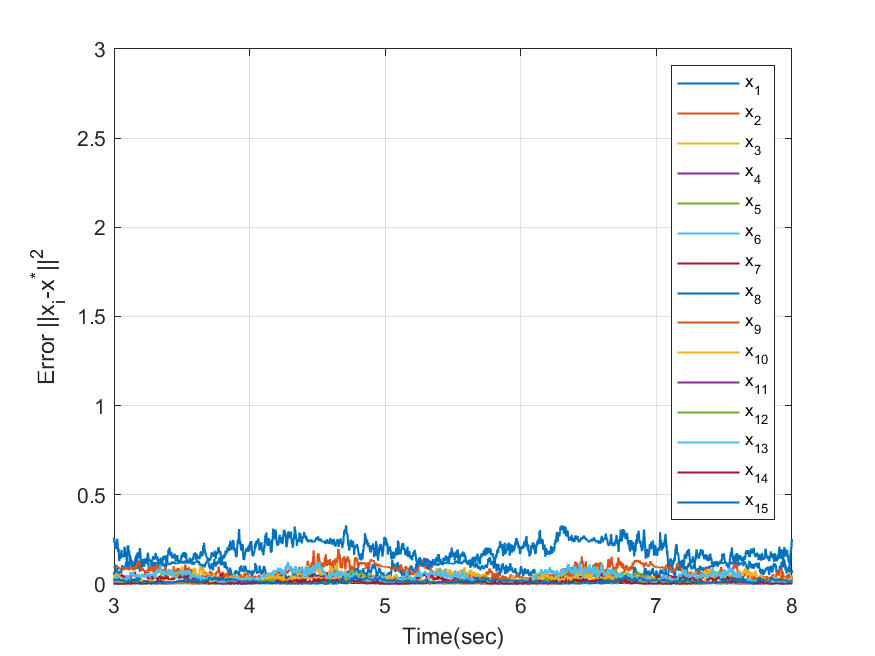}
\caption{\small The Results of Example \ref{ex2}.}
 \label{fig2}
\end{figure*}
\begin{example}  \label{ex2}
  For $n=2$, we illustrate the theoretical results with a simulation experiment with 15 agents. Let  \{(-1,-3), (0,-1), (1,2), (-2,0), (-3,4), (5,-2), (2,1), (3,-4), (-4,3), (-5,-2), (3,2), (0.5,-0.5), (-6,-5), (10,5), (-10,10)\} as the initial states and the following $\mathcal{A}$ as the adjacency matrix, which is directed, detail-balanced and strongly connected,
\renewcommand{\arraystretch}{0.7} 
\setlength{\arraycolsep}{2.5pt}
\setcounter{MaxMatrixCols}{15}
  \[
\mathcal{A} = \left[
\begin{array}{ccccccccccccccc}
0 & 1 & 1 & 0 & 0 & 3 & 0 & 1 & 1 & 0 & 0 & 3 & 0 & 0 & 3 \\
2 & 0 & 1 & 0 & 0 & 0 & 2 & 0 & 1 & 0 & 0 & 0 & 0 & 0 & 0 \\
2 & 2 & 0 & 2 & 0 & 0 & 2 & 1 & 0 & 2 & 0 & 0 & 2 & 0 & 0 \\
0 & 0 & 3 & 0 & 3 & 3 & 0 & 0 & 1 & 0 & 3 & 3 & 0 & 2 & 3 \\
0 & 0 & 0 & 3 & 0 & 3 & 0 & 0 & 0 & 2 & 0 & 3 & 2 & 0 & 0 \\
2 & 0 & 0 & 2 & 2 & 0 & 2 & 0 & 0 & 2 & 3 & 0 & 2 & 2 & 0 \\
0 & 1 & 1 & 0 & 0 & 1 & 0 & 1 & 1 & 0 & 0 & 3 & 0 & 0 & 3 \\
1 & 0 & 1 & 0 & 0 & 0 & 1 & 0 & 1 & 0 & 0 & 0 & 0 & 0 & 0 \\
2 & 2 & 0 & 2 & 0 & 0 & 2 & 2 & 0 & 2 & 0 & 0 & 2 & 0 & 0 \\
0 & 0 & 3 & 0 & 3 & 3 & 0 & 0 & 3 & 0 & 3 & 3 & 0 & 2 & 3 \\
0 & 0 & 0 & 3 & 0 & 3 & 0 & 0 & 0 & 3 & 0 & 3 & 2 & 0 & 0 \\
2 & 0 & 0 & 2 & 2 & 0 & 2 & 0 & 0 & 2 & 2 & 0 & 2 & 2 & 0 \\
0 & 0 & 1 & 0 & 1 & 0 & 0 & 0 & 1 & 0 & 1 & 1 & 0 & 2 & 3 \\
0 & 0 & 0 & 3 & 0 & 3 & 0 & 0 & 0 & 3 & 0 & 3 & 3 & 0 & 3 \\
2 & 0 & 0 & 2 & 2 & 0 & 2 & 0 & 0 & 2 & 2 & 0 & 2 & 2 & 0 \\
\end{array}
\right].
\]
Then Problem \ref{P2} is
\begin{equation}
\left\{
\begin{array}{l}
\min \mathbb{E}\left[\sum_{i\in \mathfrak{N}}f_i(t,x_i)\right]\nonumber\\
d x_{it} = u_i^d dt + \begin{pmatrix}
 0.5\sin(\pi t)  & 0 \\
 0 &  0.5\cos(\pi t)
\end{pmatrix}dB(t),\\
\mbox{subject to} \, \mathbb{E}\|x_{it}-\frac{1}{N}\sum_{j\in \mathfrak{N}}x_{jt}\|^2\leq\delta, \,\forall i\in \mathfrak{N},
\end{array}
\right.
\end{equation}
where $f_i(t,x_i)$ are shown in Table \ref{table} and the protocol is as follows, with parameters satisfying the requirements in Theorems \ref{the-estimate}-\ref{the-optimal},
\begin{align*}
u_i^d=& \sum_{j\in \mathfrak{N}}\tilde{a}_{ij}(-5sig^{0.8}(\varepsilon_{ij})-5sig^{1.2}(\varepsilon_{ij}) -3sig(\varepsilon_{ij}))\\
&- 15\nabla f_i(t,x_{it})- z_{it}^{-1}\nabla_{xt} f_i(t,x_{it})\\
z_{it}=&\zeta_{it}+H_i(t,x_{it}),\\
\dot{\zeta}_{it}=&\sum_{j\in \mathfrak{N}}\tilde{a}_{ij}(-0.5 sig^{0.8}(\hat{z}_{ij}) -0.5sig^{1.2}(\hat{z}_{ij})
-3sign(\hat{z}_{ij})),
\end{align*}
where $\varepsilon_{ij} = x_{it}-x_{jt}$, $\hat{z}_{ij}=z_{it}-z_{jt}$. Furthermore, according to the theoretical results, it can be calculated that $T_1=1.3596s$, $T_2=0.3491s$ with $\theta=0.01$.
The result is displayed in the Figure \ref{fig2}. Based on the objective function given in the Table \ref{table}, we can compute the optimal trajectory as
\begin{align*}
&x_{1t}^* = \frac{2\tanh(t) + 3\cos(\pi t)-0.5}{17.5 + 2e^{-t} + 0.5e^{-2t} + \frac{1}{t+1}}, \\
&x_{2t}^* = \frac{4\sin(\pi t)-0.5}{\frac{73}{6} + 6e^{-t} + \frac{2}{t+2} + \frac{5}{2(t+1)}},
\end{align*}
which are plotted as red curves in the first and second of the Figure \ref{fig2}. Besides, the SMASs have achieved Pfxc in probability, i.e. \[\mathbb{E}\left(\frac{1}{N}\sum_{i\in \mathfrak{N}}\left\|x_{it}-\frac{1}{N}\sum_{j\in \mathfrak{N}}x_{jt}\right\|^2\right)\leq1.51, \,\forall t\geq T_1+\tau.\]
The last figure of Figure \ref{fig2} indicates that the tracking errors of the agents with respect to the optimal trajectory $x^*_t$ satisfy $\lim_{t\rightarrow \infty}\mathbb{E}(\frac{1}{N}\sum_{i\in \mathfrak{N}}\|x_{it}-x_t^*\|^2)\leq 3.94$. Observing the Figure \ref{fig2}, It is observed that the protocol \eqref{u_d} achieves the desired results.
\begin{table}[h]
\caption{The Objective Function of Each Agent for Example \ref{ex2}.\label{table}}
  \begin{center}
  \begin{tabular}{|p{0.5cm}|p{6cm}|} 
  \hline
  $i$ & $f_i(t,x)$ \\
  \hline
  1 & $\frac{2e^{-t}+1}{4}(x_1-1)^2+\frac{e^{-t}+1}{2}(x_2-2)^2$\\
  \hline
  2 & $0.5(x_1-\tanh(t))^2+\frac{t+4}{4t+8}x_2^2$\\
  \hline
  3 & $(x_1-\sin(\pi t))^2+0.5(x_2-\sin (\pi t))^2$\\
  \hline
  4 & $0.5(x_1-\cos(\pi t))^2+0.5(x_2-\sin(\pi t))^2$\\
  \hline
  5 & $0.5(x_1-\cos(\pi t))^2+(e^{-t}+0.5)x_2^2+\sin(\pi t)$\\
  \hline
  6 & $0.5(x_1-e^{-t})^2+\frac{1}{2t+4}x_2^2+e^{-t}$\\
  \hline
  7 & $0.5(x_1-\tanh(t))^2+\frac{1}{2t+2}(x_2-1)^2-\cos(\pi t)$\\
  \hline
  8 & $0.5e^{-t}(x_1+1)^2+0.5e^{-t}(x_2+2)^2$\\
  \hline
  9 & $\frac{t+2}{2t+2}x_1^2+0.5(x_2-\cos (\pi t))$\\
  \hline
  10 & $0.5(x_1-\cos(\pi t))^2+\frac{1}{2t+2}(x_2+1)^2$\\
  \hline
  11 & $0.5(x_1+e^{-t})^2+(e^{-t}+\frac{1}{3})x_2^2$\\
  \hline
  12 & $0.5(x_1+\cos(\pi t))^2+0.5(x_2-\sin(\pi t))^2$\\
  \hline
  13 & $(x_1+\sin(\pi t))^2+0.5(x_2+\cos(\pi t))^2$\\
  \hline
  14 & $(0.25e^{-2t}+1)x_1^2+0.5(x_2-\sin(\pi t))^2$\\
  \hline
  15 & $0.5(x_1-\cos(\pi t))^2+\frac{2t+3}{4t+4}x_2^2$\\
  \hline
  \end{tabular}
  \end{center}
\end{table}
\end{example}

\section{Conclusion}
\noindent In this paper, we present a framework for solving a class of TV-OPs for the SMASs. First, we propose a centralized protocol to deal with the centralized TV-OPs. Subsequently, it is extended to the distributed case, and a distributed protocol is designed to solve this problem under a class of directed graphs. With the protocol parameters satisfying the corresponding conditions, we affirm the efficacy of both protocols through numerical simulations.

The method proposed in this paper is extendable to a wider range of TV-OPs. Future work includes distributed non-convex TV-OPs, TV-OPs with constraints, and event-triggering mechanisms to save communication resources, among others.

\section*{Acknowledgements}
This manuscript was submitted to IEEE Transactions on Cybernetics on January 20, 2025. The authors are grateful to the editors and referees for their valuable comments and suggestions.


\begin{thebibliography}{00}
\bibitem{2019-mao}
S. Mao, Z. Dong, P. Schultz, Y. Tang, K. Meng, Z. Dong, F. Qian. ``A finite-time distributed optimization algorithm for economic dispatch in smart grids,'' \textit{IEEE Transactions on Systems, Man, and Cybernetics}, vol. 51, no. 4, pp.2068--2079, 2019.

\bibitem{2004-rabbat}
M. Rabbat, N. Robert. ``Distributed optimization in sensor networks,'' in \textit{Proceedings of the 3rd International Symposium on Information Processing in Sensor Networks},  Berkeley, California, USA, pp.20--27, 2004.

\bibitem{2021-bae}
H. Bae, S. Ha, M. Kang, H. Lim, C. Min, J. Yoo. ``A constrained consensus based optimization algorithm and its application to finance,'' \textit{Applied Mathematics and Computation}, vol. 416, no. 1, 126726, 2022.

\bibitem{1986-Tsi}
J. Tsitsiklis, D. Bertsekas, M. Athans. ``Distributed asynchronous deterministic and stochastic gradient optimization algorithms,'' \textit{IEEE Transactions on Automatic Control}, vol. 31, no. 9, 803--812, 1986.

\bibitem{2009-nedic}
A. Nedi\'{c}, A. Ozdaglar. ``Distributed subgradient methods for multi-agent optimization,'' \textit{IEEE Transactions on Automatic Control}, vol. 54, no. 1, 48--61, 2009.

\bibitem{2012-lu}
J. Lu, C. Tang. ``Zero-gradient-sum algorithms for distributed convex optimization: The continuous-time case,'' \textit{IEEE Transactions on Automatic Control}, vol. 57, no. 9, pp. 2348--2354, 2012.

\bibitem{2016-lin}
P. Lin, W. Ren, Y. Song. ``Distributed multi-agent optimization subject to nonidentical constraints and communication delays,'' \textit{Automatica}, vol. 65, pp. 120--131, 2016.

\bibitem{2016-yang}
S. Yang, Q. Liu, J. Wang. ``Distributed optimization based on a multiagent system in the presence of communication delays,'' \textit{IEEE Transactions on Systems, Man, and Cybernetics: Systems}, vol. 47, no. 5, pp. 717--728, 2016.

\bibitem{2018-wang}
D. Wang, Z. Wang, M. Chen, W. Wang. ``Distributed optimization for multi-agent systems with constraints set and communication time-delay over a directed graph,'' \textit{Information Sciences}, vol. 438, pp. 1--14, 2018.

\bibitem{2015-wang}
X. Wang, Y. Hong , H. Ji. ``Distributed optimization for a class of nonlinear multiagent systems with disturbance rejection,'' \textit{IEEE Transactions on Cybernetics}, vol. 46, no. 7, pp. 1655--1666, 2015.


\bibitem{2022-liu}
P. Liu, R. Li. ``Distributed optimization for a class of uncertain nonlinear multi-agent systems with arbitrary relative degree subject to exogenous disturbances,'' \textit{International Journal of Robust and Nonlinear Control}, vol. 32, no. 8, pp. 4631--4647, 2022.

\bibitem{2023-yu} 
Z. Yu, J. Sun, S. Yu, H. Jiang. ``Fixed-time distributed optimization for multi-agent systems with external disturbances over directed networks,'' \textit{International Journal of Robust and Nonlinear Control}, vol. 33, no. 2, pp. 953--972, 2023.

\bibitem{2021-li}
S. Li, X. Nian, Z. Deng. ``Distributed optimization of second-order nonlinear multiagent systems with event-triggered communication,'' \textit{IEEE Transactions on Control of Network Systems}, vol. 8, no. 4, pp. 1954--1963, 2021.

\bibitem{2021-song}
Y. Song, J. Cao, L. Rutkowski. ``A fixed-time distributed optimization algorithm based on event-triggered strategy,'' \textit{IEEE Transactions on Network Science and Engineering}, vol. 9, no. 3, pp. 1154--1162, 2021.

\bibitem{2018-chen}
G. Chen, Z. Li. ``A fixed-time convergent algorithm for distributed convex optimization in multi-agent systems,'' \textit{Automatica}, vol. 95, pp. 539--543, 2018.

\bibitem{2020-wang}
X. Wang, G. Wang, S. Li. ``A distributed fixed-time optimization algorithm for multi-agent systems,'' \textit{Automatica}, vol. 122, 109289, 2020.

\bibitem{2021-yu}
Z. Yu, S. Yu, H. Jiang, X. Mei. ``Distributed fixed-time optimization for multi-agent systems over a directed network,'' \textit{Nonlinear Dynamics}, vol. 103, no. 1, pp. 775--789, 2021.

\bibitem{2023-garg}
K. Garg, M. Baranwal. ``Accelerating Distributed Optimization via Fixed-Time Convergent Flows,'' \textit{IFAC-PapersOnLine}, vol. 56, no. 2, pp. 1235--1240, 2023.

\bibitem{2010-wang}
J. Wang , N. Elia. ``Control approach to distributed optimization,'' in \textit{48th Annual Allerton Conference on Communication, Control, and Computing (Allerton), IEEE}, Monticello, Illinois, USA, pp. 557--561, 2010.

\bibitem{2018-nedic}
 A. Nedi\'{c}, J. Liu. ``Distributed optimization for control,'' \textit{Annual Review of Control, Robotics, and Autonomous Systems}, vol. 1, no. 1, pp. 77--103, 2018.

\bibitem{2019-yang}
T. Yang, X. Yi, J. Wu, et al. ``A survey of distributed optimization,'' \textit{Annual Reviews in Control}, vol. 47, pp. 278--305, 2019.

\bibitem{2016-rahili}
S. Rahili, W. Ren. ``Distributed continuous-time convex optimization with time-varying cost functions,'' \textit{IEEE Transactions on Automatic Control}, vol. 62, no. 4, pp. 1590--1605, 2016.

\bibitem{2017-sun}
C. Sun, M. Ye, G. Hu. ``Distributed time-varying quadratic optimization for multiple agents under undirected graphs,'' \textit{IEEE Transactions on Automatic Control}, vol. 62, no. 7, pp. 3687--3694, 2017.

\bibitem{2022-ding}
Z. Ding. ``Distributed time-varying optimization—an output regulation approach,'' \textit{IEEE Transactions on Cybernetics}, vol. 54, no. 4, pp. 2155--2165, 2022.

\bibitem{2016-simon}
A. Simonetto, A. Mokhtari, A. Koppel, G. Leus, A. Ribeiro. ``A class of prediction-correction methods for time-varying convex optimization,'' \textit{IEEE Transactions on Signal Processing}, vol. 64, no. 17, pp. 4576--4591, 2016.

\bibitem{2020-sun}
S. Sun, W. Ren. ``Distributed continuous-time optimization with time-varying objective functions and inequality constraints,'' in \textit{59th IEEE Conference on Decision and Control (CDC), IEEE}, pp. 5622--5627, 2020.

\bibitem{2024-yue}
Y. Yue, Q. Liu. ``Distributed dual consensus algorithm for time-varying optimization with coupled equality constraint,'' \textit{Applied Mathematics and Computation}, vol. 474, 128712, 2024.

\bibitem{2017-ning}
B. Ning, Q. Han , Z. Zuo. ``Distributed optimization for multiagent systems: an edge-based fixed-time consensus approach,'' \textit{IEEE Transactions on Cybernetics}, vol. 49, no. 1, pp. 122--132, 2017.

\bibitem{2024-zhu}
W. Zhu, Q. Wang. ``Distributed finite-time optimization of multi-agent systems with time-varying cost functions under digraphs,'' \textit{IEEE Transactions on Network Science and Engineering}, vol. 11, no. 1, pp. 556--565, 2024.

\bibitem{2020-li}
Z. Li, Z. Ding. ``Time-varying multi-objective optimisation over switching graphs via fixed-time consensus algorithms,'' \textit{International Journal of Systems Science}, vol. 51, no. 15, pp. 2793--2806, 2020.

\bibitem{2024-li}
W. Li, X. Wang, N. Huang. ``Finite-time/fixed-time distributed optimization for nonlinear multi-agent systems with time-varying cost function,'' \textit{Neurocomputing}, vol. 583, 127589, 2024.

\bibitem{2019-you}
X. You, C. Hua, H. Yu, X. Guan. ``Leader-following consensus for high-order stochastic multi-agent systems via dynamic output feedback control,'' \textit{Automatica}, vol. 107, pp. 418--424, 2019.

\bibitem{2024-tang}
H. Tang, T. Zhang, M. Xia. ``Distributed adaptive finite-time output feedback containment control for nonstrict-feedback stochastic multi-agent systems via command filters,'' \textit{Neurocomputing}, vol. 597, 128021, 2024.

\bibitem{D2.2}
B. Shi, J. Yuan, Y. Cheng. {\it{Fundamentals of Systems Stability and Control}}(in Chinese), Publishing House of Electronics Industry, Beijing, China, 2020.

\bibitem{L2.3}
C. Godsil, G. Royle. {\it{Algebraic Graph Theory}}, New York, NY, USA: Springer, 2001.

\bibitem{convex}
Y. Nesterov. {\it{Introductory Lectures on Convex Optimization: A Basic Course}}, New York, NY, USA: Springer, 2003.

\bibitem{2018-2}
G. Still. {\it{Lectures on Parametric Optimization: An Introduction}}, Optimization Online, 2018.

\bibitem{L2.2}
G. Hardy, J. Littlewood, G. Polya. {\it{Inequalities}}, Cambridge University Press, Cambridge, U.K., 1988.

\bibitem{L2.2+1}
G. H. Golub, C. F. Van Loan. {\it{Matrix Computations}}, fourth edition, Johns Hopkins University Press, Baltimore, America, 2013.

\bibitem{fix4}
Y. Li, X. He, D. Xia. ``Distributed fixed-time optimization for multi-agent systems with time-varying objective function,'' \textit{International Journal of Robust and Nonlinear Control}, vol. 32, no. 11, pp.  6523--6538, 2022.

\bibitem{Matrix-Analysis}
 R. A. Horn, C. R.Johnson. {\it{Matrix Analysis}}, Cambridge University Press, Cambridge, U.K., 1985.

\bibitem{fix8}
S. Parsegov, A. Polyakov, P. Shcherbakov. ``Nonlinear fixed-time control protocol for uniform allocation of agents on a segment,'' \textit{2012 IEEE 51st IEEE Conference on Decision and Control (CDC)}, pp. 7732--7737.

\bibitem{D2.4}
A. Polyakov. ``Nonlinear feedback design for fixed-time stabilization of linear control systems,'' \textit{IEEE Transactions on Automatic Control}, vol. 57, no. 8, pp. 2106--2110, 2011.

\bibitem{SSDE}
R. Khasminskii. {\it{Stochastic Stability of Differential Equations}}, New York, NY, USA: Springer, 2012.

\bibitem{2001-deng}
H. Deng, M. Krstic, R. J. Williams. ``Stabilization of stochastic nonlinear systems driven by noise of unknown covariance,'' \textit{IEEE Transactions on Automatic Control}, vol. 46, no. 8, pp. 1237--1253, 2001.

\bibitem{2023-min}
H. Min, S. Shi, S. Xu, J. Guo, Z. Zhang. ``Fixed-time lyapunov criteria of stochastic nonlinear systems and its generalization,'' \textit{IEEE Transactions on Automatic Control}, vol. 68, no. 8, pp. 5052--5058, 2023.


\bibitem{2020-he}
D. He. ``Boundedness theorems of stochastic differential systems with L\'{e}vy noise,'' \textit{Applied Mathematics Letters}, vol. 106, 106358, 2020.

\bibitem{2019-xu}
S. Xu, J. Cao, Q. Liu, L. Rutkowski. ``Optimal control on finite-time consensus of the leader-following stochastic multiagent system with heuristic method,'' \textit{IEEE Transactions on Systems, Man, and Cybernetics: Systems},'' vol. 51, no. 6, pp. 3617--3628, 2019.

\end{thebibliography}
\end{document}